\def\bbuildrel#1_#2^#3{\mathrel{\mathop{\kern 0pt#1}\limits_{#2}^{#3}}}

\def\do{\hbox to 20pt{\rightarrowfill}}

\def\N{\mathbb N}

\def\ms{\medskip}
\def\noi{\noindent}

\def\w{\thinspace\hbox{\hsize 14pt \rightarrowfill}\thinspace}

\def\0{\hbox{$\emptyset$}}

\def\s{\hbox{$\sigma$}}

\def\sub{\subseteq}

\def\H{\mathscr{H}}

\def\U{\mathscr{U}}

\def\Haus{\mathscr{H}}

\documentclass[a4paper,12pt]{amsart}

\usepackage[final]{changes} %to daje tekst po zmianach

\usepackage{amssymb}
\usepackage{mathrsfs}

\usepackage{mathabx}
\newcommand{\sk}{{\N}^{<\N}}

\newcommand{\cantor}{2^{\N}}
\newcommand{\baire}{\N^{\N}}

\newtheorem{theorem}{Theorem}[section]%chodzi o numeracje wewnątrz \section
\newtheorem{lemma}[theorem]{Lemma}

\newtheorem{proposition}[theorem]{Proposition}

\numberwithin{equation}{section}
\newtheorem{claim}[theorem]{Claim}

\newtheorem{subtheorem}{Theorem}[subsection]%chodzi o numeracje wewnątrz \subsection
\newtheorem{subproposition}[subtheorem]{Proposition}
\newtheorem{subcor}[subtheorem]{Corollary}
\newtheorem{subclaim}[subtheorem]{Claim}
\newtheorem{subremark}[subtheorem]{Remark}

\begin{document}

\title{On Borel maps, calibrated \s-ideals and homogeneity}

\author{R. Pol and P. Zakrzewski}
\address{Institute of Mathematics, University of Warsaw, ul. %Banacha 2,
02-097 Warsaw, Poland}
\email{pol@mimuw.edu.pl, piotrzak@mimuw.edu.pl}

%\thanks{The research of the second author was supported by MNiSW %Grant Nr N N201 543638.}

%\subjclass{}

%\date{June 14, 2017}


\begin{abstract}
%\added{Napisz, prosz\c e}
Let $\mu$ be a Borel measure on a compactum $X$. The main objects in this paper are \s-ideals $I(dim)$, $J_0(\mu)$, $J_f(\mu)$ of Borel sets in $X$ that can be covered by countably many compacta which are finite-dimensional, or of $\mu$-measure null, or of finite $\mu$-measure, respectively. Answering a question of J. Zapletal, we shall show that for the Hilbert cube, the \s-ideal $I(dim)$ is not homogeneous in a strong way. We shall also show that in some natural instances of measures $\mu$ with non-homogeneous \s-ideals
$J_0(\mu)$ or  $J_f(\mu)$, the completions of the quotient Boolean algebras $Borel(X)/J_0(\mu)$ or $Borel(X)/J_f(\mu)$ may be homogeneous. 

We discuss the topic in a more general setting, involving calibrated \s-ideals.

\end{abstract}

\maketitle

\section{Introduction}\label{sec:1}
The results \replaced{of }{outlined in} this paper provide more information on the topic investigated in our articles \cite{p-z-1}, \cite{p}, \cite{p-z-2}, which were strongly influenced by the work of Zapletal \cite{zap}, \cite{zap2},  Farah and Zapletal \cite{f-z} and Sabok and Zapletal \cite{s-z}.

Given a subset $E$ of a compactum (i.e., a compact metrizable space) \added{or, more generally, of a Polish (i.e., a separable completely  metrizable)} space $X$, we denote by $Bor(E)$ 
the \s-algebra of Borel sets in $E$, and 
$K(E)$ is the collection of compact subsets of $E$.

A {\sl \s-ideal 
%generated by compact sets 
on $X$} is a collection $I\sub Bor(X)$, closed under taking Borel subsets and countable unions of elements of $I$;   it is {\sl generated by compact sets} if   any element of $I$ can be enlarged to a \s-compact set in $I$. We usually assume that $X\notin I$.

A \s-ideal $I$ generated by compact sets in $X$ is {\sl calibrated} if  for any $K\in K(X)\setminus I$ and 
$K_n\in I\cap K(X)$, $n\in \N$, there is a compact set $L\sub K\setminus\bigcup\limits_{n\in\N}K_n$ not in $I$,
cf. Kechris, Louveau and Woodin \cite{k-l-w}.

Let us recall that a compactum is {\sl countable-dimensional} if it is a union of countably many zero-dimensional sets, cf. \cite{E}.

One of the main results in this paper is the following theorem.

\begin{theorem}\label{main}
	Let $I$ be a calibrated  \s-ideal on a compactum $X$ without isolated points, containing all singletons, and let $f:B\w Y$ be a Borel map from $B\in Bor(X)\setminus I$ to a compactum $Y$ without isolated points. Then
	\begin{enumerate}

\item[(i)] there exists a compact meager set $C\sub Y$ with $f^{-1}(C)\not\in I$,

\item[(ii)] if $Y$ is countable-dimensional, there is a zero-dimensional compactum $C$ in $Y$ with $f^{-1}(C)\not\in I$,

\item[(iii)] for any \s-finite nonatomic Borel measure $\mu$ on $Y$, there is a compact set $C$ in $Y$ with $\mu(C)<\infty$ and $f^{-1}(C)\not\in I$.

\end{enumerate}
	
\end{theorem}

The statement in (i) strengthens a result in \cite{p-z-2} concerning the 
``1-1 or constant" property of Sabok and Zapletal \cite{s-z}, \cite{sa} (some deep refinements of this result, in another direction, are given in the book by Kanovei, Sabok and Zapletal \cite[Section 6.1.1]{k-s-z}), cf. Section \ref{1-1 or constant}. 

%The statement in (ii) provides an answer to a question by Zapletal  \cite[Question 3.1] {zap2} (partially answered in \cite{p}).

\smallskip

To comment on (ii), let us recall the notion of homogeneity of \s-ideals introduced by Zapletal \cite{zap}, \cite{zap1}: a \s-ideal $I$ on a Polish space  $X$ is {\sl homogeneous}, if for each $E\in  Bor(X)\setminus I$ there exists a Borel map $f:X\w E$ such that $f^{-1}(A)\in I$, whenever $A\in I$.

\smallskip 

Now, (ii) implies that  the \s-ideal $I(dim)$  of Borel sets in the Hilbert cube $[0,1]^{\N}$ that can be covered by countably many finite-dimensional compacta is not homogeneous in a strong way: there are compacta $X$, $Y$ in  $[0,1]^{\N}$ not in $I(dim)$ such that for any Borel map $f:B\to Y$ on $B\in Bor(X)\setminus I(dim)$ there is a zero-dimensional compactum $C$ in $Y$ with $f^{-1}(C)\not\in I(dim)$.
Combined with a theory developed by Zapletal \cite{zap}, it shows that the forcings associated with the collections $Bor(X)\setminus I(dim)$ and $Bor(Y)\setminus I(dim)$, partially ordered by inclusion, are not equivalent. This provides an answer to a question by Zapletal \cite{zap2}, cf. Section \ref{I(dim)} for more details.

\smallskip 

In the context of homogeneity we shall discuss also \s-ideals $J_f(\mu)$ and $J_0(\mu)$ associated with  Borel measures $\mu$ on compacta $X$:  $J_f(\mu)$  ($J_0(\mu)$) is the collection of Borel sets in $X$ that can be covered by countably many compact sets of finite $\mu$-measure (of $\mu$-measure zero, respectively, cf.  \cite{ba-ju} and \cite{sa}). 

 The \s-ideal $J_0(\mu)$ is calibrated, and hence, if $\mu$ is  \s-finite and  nonatomic, (iii) shows that for any Borel map $f: B\w X$ on $B\in Bor(X)\setminus J_0(\mu)$, there is a compact set $C$ in $X$ with $\mu(C)<\infty$ and  $f^{-1}(C)\notin J_0(\mu)$, cf. Section \ref{comparing algebras}(A) for additional information.

The classical Lusin theorem shows that $J_0(\lambda)$ is not homogeneous for the Lebesgue measure $\lambda$ on $[0,1]$, and a refinement of the Lusin theorem, cf. \cite[Proposition 6.2]{p-z-2}, provides non-homogeneity of the \s-ideal $J_f(\H^1)$ associated with the 1-dimensional Hausdorff measure $\H^1$ on the Euclidean square $[0,1]^2$ (cf. Corollary \ref{Haus-inhom}).

\smallskip 

Shifting our attention from the \s-ideals
$J_0(\mu)$ and $J_f(\mu)$ to the collections $Bor(X)\setminus J_0(\mu)$ and 
$Bor(X)\setminus J_f(\mu)$, we get a different picture concerning homogeneity.

\smallskip 

Let us recall that a Borel measure $\mu$ on a compactum $X$ is  {\sl semifinite} if each Borel set of positive $\mu$-measure contains a Borel set of finite positive $\mu$-measure (\s-finite Borel measures and Hausdorff measures on Euclidean cubes are semifinite, cf. \cite{r}).

\begin{theorem}\label{quotients}
	Let $\mu$ be a 
	%semifinite
	 nonatomic Borel measure on a  compactum $X$. 
	\begin{enumerate}
		
		\item[(i)] 
		Assume that  $X\not\in J_f(\mu)$ and every Borel set $B\not\in J_f(\mu)$ contains a Borel set 
		$C\not\in J_f(\mu)$ with $\mu(C)<\infty$.
		%If $B\in Bor(X)\setminus J_0(\mu)$, 
		Then  the completion of the quotient Boolean algebra $Bor(X)/J_f(\mu)$ is homogeneous and 
		isomorphic to the completion of the quotient Boolean algebra $Bor([0,1]^2)/J_f(\Haus^1)$.

		\item[(ii)] If $\mu$ is semifinite, %If $B\in Bor(X)\setminus J_0(\mu)$, 
		then the completion of the quotient Boolean algebra $Bor(X)/J_0(\mu)$ is homogeneous and isomorphic to 	 the completion of the quotient Boolean algebra  $Bor([0,1])/J_0(\lambda)$.

	 	\end{enumerate}
	
\end{theorem}

In particular, 
%for the Lebesgue measure $\lambda$ on $[0,1]$,
 the partial order 
$Bor([0,1]^2)\setminus J_f(\Haus^1)$ is forcing homogeneous, while the \s-ideal 
$J_f(\Haus^1)$ is not homogeneous, and the same is true if $J_f(\Haus^1)$ is replaced by $J_0(\lambda)$.
It seems that examples illustrating this phenomenon did not appear in the literature, (cf.  \cite{zap}, comments following Definition 2.3.7).

Let us however remark that if $\mu$ is a \s-finite nonatomic Borel measure on a  compactum $X$ not in $J_f(\mu)$, then the \s-ideal  $J_f(\mu)$ can be homogeneous, cf. Proposition \ref{counterexample}. 

\ms

The proof of Theorem \ref{main} 
is presented in Sections \ref{proof of (i)}, \ref{proof of (ii)} and \ref{proof of (iii)}.  They are preceded by  Section \ref{sec:2} containing some preliminaries.  Our approach is similar to that in \cite{p} and \cite{p-z-2}, an essential difference being that we shall analyze compact-valued functions $\widecheck{f}_U:Y\w K(X)$ associated with $f^{-1}$ rather than functions
  $\widehat{f}_U:\overline{U}\w K(Y)$ considered in \cite{p} or \cite{p-z-2}, associated with $f$.
 
 Theorem \ref{quotients} is based on results of Oxtoby \cite{Ox} and its proof is presented in Section \ref{homo}.

 \section{Preliminaries} \label{sec:2} 
 Our notation is standard and mostly agrees with \cite{k}. In particular,
 
 \begin{itemize}
 %	\item $\cantor$ and $\baire$ are the Cantor and the Baire space, respectively,
 	
 	\item $\N=\{0,1,\ldots\}$,
 	
 	\item $\sk$ is the family of all finite sequences of natural numbers,
 	
 	\item in a given metric space: $\hbox{diam}(A)$  is the diameter of  $A$, $B(x,r)$ is the open $r$-ball  centered at $x$ and $B(A,\varepsilon)$ is the open $\varepsilon$-ball around $A$.

 \end{itemize}
 
 The terminology concerning Boolean algebras
 agrees with \cite{ko}.
 
 As in our earlier work on this topic, the key element of our reasonings are generalized Hurewicz systems, cf. \cite[2.4]{p-z-2}; such systems were introduced in some special cases by W. Hurewicz \cite{hu}, and significantly developed by S. Solecki \cite{s1} in connection with \s-ideals generated by closed sets.

Let $X$ be a compactum without isolated points. In this paper by a {\sl generalized Hurewicz system} we shall mean (adopting  a slightly more restrictive definition  than in \cite[2.4]{p-z-2}) a pair 
$(U_s)_{s\in\sk}$, $(L_s)_{s\in\sk}$ of families of subsets of $X$ with the following properties, where $G$ is  a given non-empty $G_{\delta}$-set in  $X$, the diameters are with respect to a fixed complete metric on $G$  and the closures are taken in $X$:
 \begin{itemize}
 	\item $U_s\sub G$ is relatively open, non-empty and $\hbox{\rm diam} (U_s) \leq 2^{-length(s)}$,
 	
 	\item $\overline{U_s}\cap \overline{U_t}=\emptyset$ for distinct $s,\ t$ of the same length,
 	
 	\item $\overline{U_{s \hat\ i}}\cap G\sub U_s$,

 	\item $L_s \sub \overline{U_s}$ is compact,
 	
 	\item $L_s\cap \overline{U_{s  \hat\ i}} = \emptyset$,
 	
 	\item $L_s=\bigcap_{j}\overline{\bigcup_{i>j}U_{s\hat\ i}}$,

 	\item  $\lim_i\hbox{\rm diam} (U_{s\hat\ i})=0$.
 	%and $\hbox{\rm diam} (U_s) \leq 2^{-length(s)}$,
 \end{itemize} 
 
 If a pair
 $(U_s)_{s\in\sk}$, $(L_s)_{s\in\sk}$
 is a generalized Hurewicz system, then  
 $$P = \bigcap_n \bigcup \{U_s: \hbox{length(s)} = n\}$$
 is the $G_\delta$-subset of $G$ (actually, a copy of the irrationals) {\sl determined by the system} and we have,  cf. \cite[2.4]{p-z-2},
  \begin{enumerate}
 \item[(1)] $\overline{P} = P \cup \bigcup\{L_s: s\in\sk\}.$ 
  \end{enumerate}
  Moreover, 
   \begin{enumerate}
  \item[(2)]  if $V$ is a non-empty relatively open subset of $P$, then $\overline{V}$ contains $L_s$ with arbitrarily  long $s\in\sk$. 
   \end{enumerate}

 We shall use the generalized Hurewicz systems in the following situation.

 Let $I$ be a calibrated  \s-ideal on a compactum $X$ without isolated points, containing all singletons, and let $f:B\w Y$ be a Borel map from $B\in Bor(X)\setminus I$ to a compactum $Y$ without isolated points. 
 
 Moreover, suppose that  $G\sub B$ is a non-empty, $G_{\delta}$-set  in  $X$ such that

  \begin{enumerate}
 \item[(3)] $V\not\in I$ for any non-empty relatively open set $V$ in $G$,
 
 \item[(4)] $f|G: G\w Y$ is continuous.
  \end{enumerate}
 Such a set $G$ can always be found by a theorem of Solecki \cite{s1}.

Given a non-empty relatively open set $U$ in $G$   we shall consider the map $\widecheck{f}_U:Y\w K(\overline{U})$ defined by
  
   \begin{enumerate}
  \item[(5)]
  $\widecheck{f}_U(y)=\bigcap_n \overline{f^{-1}(B(y,\frac{1}{n}))\cap U}$,
   \end{enumerate}
  $B(y,r)$ being the open $r$-ball centered at $y$, with respect to a fixed metric on $Y$. 
  
  In other words, $x \in\widecheck{f}_U(y)$ if and only if there is a sequence $(x_n)$ of elements of $U$ such that $\lim_n x_n=x$ and $\lim_n f(x_n)=y$. 
  
  Notice that for $y\not\in \overline{f(U)}$, $\widecheck{f}_U(y)=\emptyset$. The map $\widecheck{f}_U$ is upper-semicontinuous so, in particular, the set $\widecheck{f}_U[E]$ defined by
  $$
  \widecheck{f}_U[E]=\bigcup\{\widecheck{f}_U(y): y\in E\}
  $$
  is compact, whenever $E$ is compact. Let us also notice that $x\in \widecheck{f}_U(f(x))$ for any $x\in U$ and hence, $\widecheck{f}_U[Y]$ being compact,
  
   \begin{enumerate}
  \item[(6)] $\overline{U}=\widecheck{f}_U[Y].$
  \end{enumerate}

% In this setting 
  Functions $\widecheck{f}_U$, associated with  $f$, $G$ fixed and $U$ varying over non-empty open subsets of $G$, will be used to define  generalized Hurewicz systems %in a way that gives
  providing some  control simultaneously over  sets determined by the systems and their images  under $f$. 
  %In the above setting
  This is explained by  the following lemma where we gathered some observations vital for the proof of Theorem \ref{main}.

 \begin{lemma}\label{main lemma}
 %	Let $I$ be a calibrated  \s-ideal on $X$, containing all singletons. 
  Assume that $I$ is a calibrated  \s-ideal on a compactum $X$ without isolated points, containing all singletons, and let $f:B\w Y$ be a Borel map from $B\in Bor(X)\setminus I$ to a compactum $Y$ without isolated points. Moreover, suppose that  $G\sub B$ is a non-empty, $G_{\delta}$-set  in  $X$ satisfying conditions (3) and (4).
  	
 	\smallskip 
 	
 {\bf (A)}	 Let $U$ be a non-empty relatively open set in $G$ and assume that 
 	  $L\sub \overline{U}$ and $M\sub\overline{f(U)}$ are compacta such that
 	 
 	 	\begin{enumerate}
 	 
 	 \item[(A1)] $L$ is boundary in $\overline{U}$ and $L\notin I$,
 	 
 	 \item[(A2)]  $f^{-1}(M)\in I$,

 	\item[(A3)] $L\sub \widecheck{f}_U[M]$.
 	 
 	 	\end{enumerate}	
 
 Then there exist   
  nonempty relatively open subsets $V_i$ of $U$ such that
 	%\begin{itemize}
 	\begin{enumerate}
 		\item[(A4)]
 		$\overline{V_i}\cap G\sub U$, 
 		
 		\item[(A5)]
 		$\overline{V_i}$ are pairwise disjoint and disjoint from $L$,
 		
 		\item[(A6)] $L=\bigcap\limits_{n}\overline{\bigcup\limits_{i\geq n}V_i}$,  
 		
 		\item[(A7)] 
 		$\overline{f(V_i)}$ are pairwise disjoint and disjoint from $M$,

 		\item[(A8)] 
 		$\lim\limits_{i\to\infty}\hbox{\rm diam} (V_i )=0$ and $\lim\limits_{i\to\infty}\hbox{\rm diam} (f(V_i))=0$ with respect to fixed metrics on $X$ and $Y$, respectively,

 		\item[(A9)] $\bigcap\limits_{n}\overline{\bigcup\limits_{i\geq n}f(V_i)}\sub M$.

 	\end{enumerate}

 	\smallskip 
 	
 	 {\bf (B)}   Let   $(J_s)_{s\in\sk}$ be a family of hereditary collections of closed subsets of $Y$.
 	Assume that for every non-empty relatively open set  $U$  in $G$ and each $s$, there exist compacta $L$ and $M\in J_s$ with  properties  (A1)--(A3). Then 
 there exists a generalized Hurewicz system $(U_s)_{s\in\sk}$, $(L_s)_{s\in\sk}$  with an associated family  $(M_s)_{s\in\sk}$ such that the following additional conditions are satisfied for each $s\in\sk$:
 	
 	\begin{enumerate}
 	%	\item[(12)]	$L_s=\bigcap\limits_{n}\overline{\bigcup\limits_{i\geq n}U_{s\hat\ i}}$,
 		
 		\item[(B1)]
 		$L_s\notin I$,

 		\item[(B2)] $M_s\in K(\overline{f(U_s)})\cap J_s$,

 		\item[(B3)]
 		$\overline{f(U_{s\smallfrown i})}$ are pairwise disjoint and disjoint from $M_s$,

 		\item[(B4)]
 		$\lim\limits_{i\to\infty}\hbox{\rm diam} (U_{s\smallfrown i} )=0$ and $\lim\limits_{i\to\infty}\hbox{\rm diam} (f(U_{s\smallfrown i}))=0$ with respect to fixed metrics on $X$ and $Y$, respectively,
 		
 			\item[(B5)]
 			$M_s= \bigcap\limits_{n}\overline{\bigcup\limits_{i\geq n}f(U_{s\smallfrown i})}$,
 	%	\item[(17)] 	$M_s\sub \overline{\overline{f(U_s)}\setminus\overline{\bigcup\limits_{i}f(U_{s\smallfrown i})}}$.

 	\end{enumerate}

 	 {\bf (C)} If $P\sub G$  is
 	 the  set  determined by the system from part (B), then
 	 \begin{enumerate}
 	 	
 	 	\item[(C1)] $\overline{P} = P \cup \bigcup\{L_s: s\in\sk\}$,
 	 	
 	 	\item[(C2)]  each non-empty relatively open subset of $\overline{P}$ contains some $L_s$ (with arbitrarily  long $s\in\sk$), 
 	 	
 	 	\item[(C3)] $P\notin I$,
 	 	
 	 	\item[(C4)]  $\overline{f(P)} = f(P)\cup \bigcup\{ M_s: s\in\sk\}$,
 	 	
 	 	\item[(C5)]  each non-empty relatively open subset of $\overline{f(P)}$ contains some $M_s$ (with arbitrarily  long $s\in\sk$).

 	 \end{enumerate}

 \end{lemma}
 
\begin{proof}
  
 In order to prove part (A), 
 let us fix a countable dense set in $L$ and list its elements, repeating each point infinitely many times, as $a_0,a_1,\ldots$.
 
We shall choose inductively non-empty relatively open sets $V_i$ in $U$ such that, 
 
 \smallskip 
 
 \begin{enumerate}

 \item[(7)] $V_i\sub B(a_i,\frac{1}{i+1}),\ \overline{V_i}\cap G\sub U,\ diamf(V_i)\leq \frac{1}{i+1},\ f(V_i)\sub B(M,\frac{1}{i+1}),$

% \smallskip 
 
\item[(8)] $\overline{V_i}\sub \overline{U}\setminus 
 (L\cup \bigcup_{j<i}\overline{V_j}),\ 
 \overline{f(V_i)}\sub \overline{f(U)}\setminus 
 (M\cup \bigcup_{j<i}\overline{f(V_j)})$.
\end{enumerate}
% \smallskip 
 
 Suppose that $V_j$, $j<i$, are already defined, where $V_0=\emptyset$.
 
 Since $a_i\in L$, by (A3) we have  $a_i \in \widecheck{f}_U(b_i)$ for some $b_i\in M$. Let $\delta_i < \frac{1}{i+1}$ be such that 
 \begin{enumerate}
 
 \item[(9)] $B(a_i,\delta_i)\cap \bigcup_{j<i}\overline{V_j}=\emptyset,\ B(b_i,\delta_i)\cap \bigcup_{j<i}\overline{f(V_j)}=\emptyset$.
 
 \end{enumerate}
 
 From (5),
 \begin{enumerate}
 \item[(10)] $W= B(a_i,\delta_i)\cap f^{-1}(B(b_i,\delta_i))\cap U\neq\emptyset$,
\end{enumerate}
% \noi
 and by (4), $W$ is relatively open in $U$. Since $L$ is boundary in $\overline{U}$, $W\setminus L\neq\emptyset$ and by (3), $W\setminus L\notin I$. Since, cf. (A2), $f^{-1}(M)\in I$, we can pick $c\in W\setminus (L\cup f^{-1}(M))$. Then $f(c)\in B(b_i,\delta_i)\setminus M$, cf. (10), and appealing again to continuity of $f$, we get a relatively open neighbourhood $V_i$ of $c$ in $U$ with $\overline{V_i}\sub B(a_i,\delta_i)\setminus L$, $\overline{V_i}\cap G\sub U$ and  $\overline{f(V_i)}\sub B(b_i,\delta_i)\setminus M$. By (9), $V_i$ satisfies (8).
 
 This completes the inductive construction. 
 It is now easy to see that requirements (A4)--(A9) of part (A) are met.
 
 \smallskip 
 
 Having  checked part (A), we can use it subsequently to define inductively a generalized Hurewicz system in $X$ with properties (B1)--(B4) and property (B5) replaced by, cf. (A9),
 $$
  	 \bigcap\limits_{n}\overline{\bigcup\limits_{i\geq n}f(U_{s\smallfrown i})}\sub M_s.
  	 $$
  Taking into account that $J_s$ is hereditary,  to secure (B5), it suffices to replace $M_s$ by
  $\bigcap\limits_{n}\overline{\bigcup\limits_{i\geq n}f(U_{s\smallfrown i})}$.
 This completes the proof of part (B). 	 
 
 \smallskip 
 
 To prove part (C), first note that properties (C1), (C2) hold for any generalized Hurewicz system considered in this paper and (C3) follows from (B1) and (C2) by a Baire category argument.

 We proceed to the proof of (C4). The inclusion
 $$
 f(P) \cup \bigcup\{ M_s: s\in\sk\}\sub\overline{f(P)}
 $$ 
 can be easily justified with the help of (B4) and (B5) combined with the observation that 
 $U_s\cap P\neq \emptyset$ for each $s$.
 
 To prove the opposite inclusion, first note that, by (B5), for each $s$ we have:
 $$
 \overline{f(P\cap U_s)}=\overline{\bigcup_i f(P\cap U_{s\smallfrown i})}\sub \bigcup_i\overline{f(P\cap U_{s\smallfrown i})}\cup M_s.
 $$
 
 %This immediately implies that $f(P) \cup \bigcup_t M_t\sub\overline{f(P)}$, so we only have to check the opposite inclusion.  
 So assume that $y\in \overline{f(P)}\setminus \bigcup \{ M_s: s\in\sk\}$ and notice  that for each $s$, if $y\in \overline{f(P\cap U_s)}$, then there is (precisely one, cf. (B3)) $i$ such that  $y\in \overline{f(P\cap U_{s\smallfrown i})}$. Using the fact that $y\in \overline{f(P\cap U_{\emptyset})}$ (recall that $P\sub U_{\emptyset}$) this allows us to construct inductively a sequence 
 $z\in \baire$ with
 $$
 y\in \overline{f(P\cap U_{z|n})}\ \mbox{for each}\ n\in\N.
 $$
 
 It follows, by the continuity of $f$,   that if $x\in P$ is the unique element of  $\bigcap_n U_{z|n}$, then $y=f(x)$, which shows  that $y\in f(P)$ completing the proof of (C4).
 \smallskip 
 
 Finally, (C5) can be easily justified with the help of (B5) and (B4).

 \end{proof}

\section{Proof of Theorem \ref{main} \hbox{\rm (i)}}\label{proof of (i)}

Striving for a contradiction, let us assume that for any meager set $C$ in $Y$, $f^{-1}(C)\in I$. In particular, since $Y$ has no isolated points, it follows that $f^{-1}(y)\in I$ for any $y\in Y$.

Using a theorem of Solecki \cite{s1}, we can find a non-empty $G_{\delta}$-set $G$ in $X$ such that $G\sub B$,
 \begin{enumerate}
\item[(1)] $V\notin I$ for any non-empty relatively open $V$ in $G$, 

\item[(2)] $f|G: G\w Y$ is continuous.
 \end{enumerate}
%Given a nonempty relatively open set $U$ in $G$ we shall consider the map $\widecheck{f}_U:Y\w K(\overline{U})$ defined by

%$$ (3)\qquad \widecheck{f}_U(y)=\bigcap_n\overline{f^{-1}(B(y,\frac{1}{n}))\cap U},$$
%$B(y,r)$ being the $r$-ball centered ay $y$, with respect to a fixed metric on $Y$.

%The maps $\widecheck{f}_U$ are used to justify the following claim, cf. \cite{p}, \cite{p-z-2}.
%Our idea is to 
We shall apply Lemma \ref{main lemma}, 
%so a key element of our reasoning is 
and to that end, we shall first establish the following fact.

\ms

\begin{claim}\label{claim for (i)}

%{\bf Claim.}
 Let $U$ be a  non-empty relatively open set in $G$. Then there exist   compacta   $L\sub \overline{U}$ and $M\sub\overline{f(U)}$ such that
 
 (3) $L$ is boundary in $\overline{U}$ and $L\notin I$,
 
 (4) $M$ is boundary in 
 %$Y$, 
  $\overline{f(U)}$,

 (5) $L\sub \widecheck{f}_U[M]$.
 
 \end{claim}
 To prove the claim,
 first note  that $\overline{f(U)}$ has no isolated points. For suppose that $y$ is an isolated point in $\overline{f(U)}$. Then, by the continuity of $f$, $f^{-1}(y)$ contains a non-empty relatively open subset of $G$ which, by (1), implies that $f^{-1}(y)\notin I$, contradicting our assumptions.
 
 \smallskip 
 
 Now let us fix a countable set $D$ dense in 
 	%$Y$.
 	 $\overline{f(U)}$. 
 	We shall consider two cases.
 	
 	\ms
 	
 	{\sl Case 1. There exists $d\in D$ with $\widecheck{f}_U(d)\notin I$.}
 	
 	\ms
 	
 	Then, since all singletons of $\widecheck{f}_U(d)$ are in $I$ and $I$ is calibrated, there exists a boundary in $\overline{U}$ compactum $L\sub \widecheck{f}_U(d)$ not in $I$, and we let $M=\{d\}$.
 	
 	\ms
 	
 	{\sl Case 2. For all $d\in D$, $\widecheck{f}_U(d)\in I$.}
 	
 	\ms
 	
 	Then, $I$ being calibrated and containing all singletons of $X$, we have a boundary compactum 
 	$L\sub \overline{U}\setminus \bigcup_{d\in D}\widecheck{f}_U(d)$, $L\notin I$. Let
 	$$
 	M=\{y\in Y: \widecheck{f}_U(y)\cap L\neq \emptyset \}.
 	$$
 	
 	The compactum $M\sub \overline{f(U)}$ is disjoint from $D$, hence boundary in 
 	%$Y$,
 	 $\overline{f(U)}$,
 	 and we have, cf.  (6) in Section \ref{sec:2}, $L\sub \widecheck{f}_U[M]$, which completes the proof of the claim.
 	
\ms

 Having verified the claim, we shall modify the proof of Lemma \ref{main lemma} to get for any non-empty relatively open set $U$ in $G$  a sequence $(V_i)$ of non-empty relatively open subsets  of $U$ with properties (A4)--(A9)  and the following additional property
%\begin{itemize}
	\begin{enumerate}
%	\item[(i)]	 $\overline{V_i}\cap G\sub U$, 
	
%	\item[(ii)]	 $\overline{V_i}$ are pairwise disjoint and disjoint from $L$,
	
%\item[(iii)] $\bigcap\limits_{n}\overline{\bigcup\limits_{i\geq n}V_i}=L$, ,  

%\item[(iv)] $\overline{f(V_i)}$ are pairwise disjoint and disjoint from $M$,

%  \item[(v)] $\lim\limits_{i\to\infty}\hbox{\rm diam} (V_i )=0$ and $\lim\limits_{i\to\infty}\hbox{\rm diam} (f(V_i))=0$.

%\item[(vi)] $\bigcap\limits_{n}\overline{\bigcup\limits_{i\geq n}f(V_i)}=M$, 

\item[(6)]
$M\sub \overline{\overline{f(U)}\setminus\overline{\bigcup\limits_{i}f(V_i)}}$, 
%$M\sub \overline{Y\setminus\overline{\bigcup\limits_{i}f(V_i)}}$,
\end{enumerate}
%\end{itemize}

%Having defined $L$ and $M$ satisfying (5) and (6), we shall choose inductively non-empty relatively open sets $V_i$ in $U$ as follows.

%It follows, 
 Namely, since  $M$ is boundary in $\overline{f(U)}$ and 
 %$Y$
 $\overline{f(U)}$
  has no isolated points,  we can enlarge $M$ 
%(adding countably many sequences converging to points in a dense subset of $M$)
to a compactum $M^*\sub \overline{f(U)}$
such that
\begin{enumerate}
\item[(7)] $M^*$ is boundary in 
%$Y$
$\overline{f(U)}$
 and $M\sub \overline{M^*\setminus M}$.
 \end{enumerate}

To get $M^*$, we fix a countable set $C$ dense in $M$, and then we pick subsequently points $d_n$ in 
$\overline{f(U)}\setminus M$ so that $d_n\in B(M,\frac{1}{n+1})$ and each point in $C$ is the limit of a subsequence of $(d_n)_{n\in\N}$. Then we let $M^*= M\cup \{d_n:n\in\N \}$.
\smallskip 

Having defined $M^*$ satisfying (7),  we proceed as in the proof of part (A) of Lemma \ref{main lemma} 
and using the fact that  our assumptions yield $f^{-1}(M^*)\in I$ we can choose inductively non-empty relatively
 open sets $V_i$ in $U$ such that
\begin{enumerate}
\item[(8)] $V_i\sub B(a_i,\frac{1}{i+1}),\ \overline{V_i}\cap G\sub U,\ diamf(V_i)\leq \frac{1}{i+1},\ f(V_i)\sub B(M,\frac{1}{i+1}),$

\item[(9)] $\overline{V_i}\sub \overline{U}\setminus 
(L\cup \bigcup_{j<i}\overline{V_j}),\ 
\overline{f(V_i)}\sub \overline{f(U)}\setminus 
(M^*\cup \bigcup_{j<i}\overline{f(V_j)})$.
\end{enumerate}

Then requirements (A4)--(A9)  of Lemma \ref{main lemma} are still met. In particular, cf. (8),  $\bigcap\limits_{n}\overline{\bigcup\limits_{i\geq n}f(V_i)}\sub M$ which, combined with (9), guarantees that 
$\overline{\bigcup_i f(V_i)}$  is disjoint from $M^*\setminus M$. However, by (7), the latter set 
contains $M$ in its closure which justifies (6).

\ms

We can now define a generalized Hurewicz system $(U_s)_{s\in\sk}$, $(L_s)_{s\in\sk}$  with the associated family  $(M_s)_{s\in\sk}$
satisfying  for each $s\in\sk$ conditions (B1)--(B5) (with $J_s$ being the \replaced{the collection of meager compacta }{\s-ideal of meager sets} in $Y$, see part (B) of Lemma \ref{main lemma}) and the following additional condition

\begin{enumerate}
%	\item[(12)] 	$L_s=\bigcap\limits_{n}\overline{\bigcup\limits_{i\geq n}U_{s\hat\ i}}$,
	
%	\item[(13)]	$L_s\notin I$,
	
%	\item[(14)]	$M_s= \bigcap\limits_{n}\overline{\bigcup\limits_{i\geq n}f(U_{s\smallfrown i})}$,
	
%	\item[(15)] 	$\overline{f(U_{s\smallfrown i})}$ are pairwise disjoint and disjoint from $M_s$,

%	\item[(16)] 	$\lim\limits_{i\to\infty}\hbox{\rm diam} (U_{s\smallfrown i} )=0$ and $\lim\limits_{i\to\infty}\hbox{\rm diam} (f(U_{s\smallfrown i}))=0$.

\item[(10)]
	$M_s\sub \overline{\overline{f(U_s)}\setminus\overline{\bigcup\limits_{i}f(U_{s\smallfrown i})}}$.

\end{enumerate}

These conditions guarantee that the set $P$  determined by this system not only has  properties (C1)--(C5) (see part (C) of Lemma \ref{main lemma}) but satisfies the following one as well

\begin{enumerate}
	
\item[(11)] \mbox{For each $s$, }
$M_s\sub \overline{Y\setminus\overline{f(P)}}$. 
\end{enumerate}
\smallskip 

To see this,
it suffices to prove,  by (10) and (C4), that for each $s$
$$
(\overline{f(U_s)}\setminus\overline{\bigcup\limits_{i}f(U_{s\smallfrown i})})
\cap( f(P) \cup \bigcup\{ M_s: s\in\sk\} ) =\emptyset.
$$

%$$\overline{f(P)}\sub\overline{\bigcup\limits_{i}f(U_{s\smallfrown i})}.$$

%so it suffices to prove that
%$$\overline{f(U_s)}\setminus\overline{\bigcup\limits_{i}f(U_{s\smallfrown i})}\sub Y\setminus\overline{f(P]}.$$

So fix $s\in\sk$ and let $y\in \overline{f(U_s)}\setminus\overline{\bigcup\limits_{i}f(U_{s\smallfrown i})}$. 
%In view of 0. it is enough  to show that  $y\not\in f(P] \cup \bigcup_t M_t$

\added{Striving for a contradiction} suppose first that $y\in f(P)$. Let $k=length(s)$.
Since
$$f(P)\sub \bigcup\{ f(U_t):\ length(t)=k \}$$
where the sets $f(U_t)$ have pairwise disjoint closures, \added{the fact that}  $y\in \overline{f(U_s)}\cap f(P)$ implies that $y\in f(U_s)$.
Consequently, since
$$f(P)\cap f(U_s)\sub f(P\cap U_s) \sub \bigcup_i f(U_{s\smallfrown i}),$$
we conclude that $y\in  \bigcup_i f(U_{s\smallfrown i})$, contrary to the assumption that $y\not\in \overline{\bigcup\limits_{i}f(U_{s\smallfrown i})}$.

\replaced{Next, if $y\in M_t$ for some $t\in\sk$, then a contradiction can be easily reached by considering the four 
	%possibilities of how 
	%$t$ is related to $s$: 
	mutual positions of $t$ and $s$ (namely, $s=t$, $s\subsetneq t$, $t\subsetneq s$, $s$ and $t$ are incompatible). }{
Suppose next that $y\in M_t$ for a certain $t\in\sk$ and let us consider the following cases:
%\end{document}
%\begin{enumerate}
	%\smallskip 
		{\sl Case 1}. $s=t$.
	Then $y\in M_s=\bigcap_n\overline{\bigcup_{i\geq n} f(U_{s\smallfrown i}) }$, contrary to the fact that 
	$y\not\in \overline{\bigcup\limits_{i}f(U_{s\smallfrown i})}$.
	%\smallskip 
		{\sl Case 2}. $s\subsetneq t$. 
		Then $y\in M_t\sub \overline{f(U_t)}\sub  \overline{f(U_{t|(k+1)})}$ and $t|(k+1)=s\smallfrown i$ for $i=t(k)$, which leads to a contradiction as in Case 1.
		%\smallskip 
		{\sl Case 3}. $t\subsetneq s$.
		 Then $y\in M_t$ and $M_t \cap \overline{f(U_{t\smallfrown i})}=\emptyset$ for each $i$, contrary to the fact that $y\in\overline{f(U_s)}\sub \overline{f(U_{t\smallfrown i})}$ for $i=s(length(t))$.
	%
	%\smallskip 
			{\sl Case 4}. $s$ and $t$ are incompatible.
		 Then $y\in M_t \sub \overline{f(U_t)}$ and $y\in  \overline{f(U_s)}$, contrary to the fact that $\overline{f(U_t)}\cap \overline{f(U_s)}=\emptyset$.}
	
%\end{enumerate}
\smallskip 
Having justified (11), \added{let us} note that combined with \replaced{(C5) }{(C4)} it  implies that $\overline{f(P)}$ has empty interior in $Y$. But on the other hand, $P\notin I$ cf. (C3). In effect, for the meager compactum $C=\overline{f(P)}$ we have $f^{-1}(C)\notin I$ and this contradiction with our assumptions
completes the proof of part (i) of Theorem \ref{main}.~\qed

\section{Proof of Theorem \ref{main} \hbox{\rm (ii)} }\label{proof of (ii)} The reasoning in this case goes along similar lines as for Theorem \ref{main}(i).

Striving for a contradiction, suppose that for any zero-dimensional compactum $C$ in $Y$, $f^{-1}(C)\in I$. 
%In particular, since $Y$ has no isolated points, it follows that $f^{-1}(y)\in I$ for $y\in Y$.

 The compactum $Y$ being countable-dimensional, $Y$ has defined the small inductive transfinite dimension ${\rm ind}\ Y$, see \cite[Theorem 7.1.9]{E}. Let $Y'$ be a compactum in $Y$ such that $f^{-1}(Y')\not\in I$ with minimal transfinite dimension. Replacing $Y$ by $Y'$ and $B$ by  $f^{-1}(Y')$, we can assume that  $f^{-1}(K)\in I$ for any compactum $K$ in $Y$ with  ${\rm ind}\ K< {\rm ind}\ Y$.

Let us choose a base for the topology of $Y$ whose elements have boundaries $K_0,K_1,\ldots$ with 
 ${\rm ind}\ K_i< {\rm ind}\ Y$. Then
 \begin{enumerate}

 \item[(1)] $f^{-1}(K_i)\in I$ for any $i$ and 
 $H=Y\setminus \bigcup\limits_i K_i$ is zero-dimensional.
\end{enumerate}

 We have $B\setminus \bigcup\limits_i f^{-1}(K_i)\notin I$ and using a theorem of Solecki \cite{s1}, we can find a non-empty $G_{\delta}$-set $G$ in $X$, $G\sub B\setminus \bigcup\limits_i f^{-1}(K_i)$, such that 
$V\notin I$ for any non-empty relatively open $V$ in $G$ and, cf. (1), 
 \begin{enumerate}
\item[(2)] $f|G: G\w H$ is continuous.
\end{enumerate}

A key element of our reasoning is the 
 following counterpart of %the claim in Section \ref{proof of (i)}.
Claim \ref{claim for (i)}.

\ms 

\begin{claim}\label{claimfor(ii)}

	%{\bf Claim.}
	Let $U$ be a  non-empty relatively open set in $G$. Then there exist   compacta   $L\sub \overline{U}$ and $M\sub\overline{f(U)}$ such that
	
\begin{enumerate}
\item[(3)] $L$ is boundary in $\overline{U}$ and $L\notin I$,
	
	\item[(4)] $M$ is zero-dimensional, 
	%in $\overline{f(U)}$,

	\item[(5)] $L\sub \widecheck{f}_U[M]$.
\end{enumerate}	
\end{claim}

%To verify this claim, we use the fact that $Y$ being countable dimensional, $Y$ has defined the small inductive transfinite dimension, and hence we can concentrate on a compactum $Y'\sub Y$ with the minimal transfinite dimension for which $f^{-1}(Y')\notin I$.

In order to prove the claim,  let, cf. (1),
\begin{enumerate}
\item[(6)] $Z=\overline{f(U)}$, $E_i=Z\cap K_i$
\end{enumerate}
and consider two cases.

\ms

{\sl Case 1. There exists $i$ such that $\widecheck{f}_U[E_i]\notin I$.}

\smallskip 

Then, let $S$ be a compactum in $E_i$ such that  $\widecheck{f}_U[S]\notin I$ with minimal possible transfinite dimension ${\rm ind}$.

Considering, as we did before, a base in $S$ whose elements have boundaries $S_0,S_1,\ldots$ with 
 ${\rm ind}\ S_i< {\rm ind}\ S$, we have that $\widecheck{f}_U[S_i]\in I$ and 
 $T=S\setminus \bigcup\limits_i S_i$ is zero-dimensional. Since $I$ is calibrated and $\widecheck{f}_U[S]\notin I$, there is a compactum 
 $L\sub \widecheck{f}_U[S]\setminus \bigcup\limits_i \widecheck{f}_U[S_i]$ not in $I$ and (since the sigletons of $X$ belong to $I$), we may demand that $L$ is boundary in $\overline{U}$.
 Then $M=\{y\in S: \widecheck{f}_U(y)\cap L\neq \emptyset \}$
is a compact subset of $T$, hence zero-dimensional, and $L\sub \widecheck{f}_U[M]$.
\ms

{\sl Case 2. For all $i$,  $\widecheck{f}_U[E_i]\in I$.}

\smallskip 

Then, as in Case 1, we can pick a boundary in 
$\overline{U}$ compactum  $L\sub \overline{U}\setminus \bigcup\limits_i \widecheck{f}_U[E_i]$ not in $I$ and since the compactum  $M=\{y\in Z: \widecheck{f}_U(y)\cap L\neq \emptyset \}$
is contained in $H$, cf. (1) and (6), $M$ is zero-dimensional and $L\sub \widecheck{f}_U[M]$.

\ms

%Let us fix a metric on $Y$.
Having justified the claim, we can use Lemma \ref{main lemma} to define a generalized Hurewicz system $(U_s)_{s\in\sk}$, $(L_s)_{s\in\sk}$  with the associated family  $(M_s)_{s\in\sk}$ 
satisfying  for each $s\in\sk$ conditions (B1)--(B5) with  $J_s$ being the collection of zero-dimensional compacta in $Y$, see part (B) of Lemma \ref{main lemma}.
% and the following additional condition
%\begin{enumerate}
%\item[(7)]$\hbox{\rm diam} (\overline{f(U_s)})<\frac{1}{2^{length(s)}}$.
%\end{enumerate}

These conditions guarantee that the set $P$  determined by this system apart from  properties (C1)--(C5) (which follow from part (C) of Lemma \ref{main lemma}) satisfies also the following one

\begin{enumerate}
\item[(8)] The compactum $\overline{f(P)}$ contains no non-trivial continuum.
\end{enumerate}

%To see this, let us notice that if $a\in f(P)$ and $b\in \overline{f(P)}\setminus f(P)$, then there exists a clopen in $\overline{f(P)}$ set containing $a$ and missing $b$.

%Indeed, let $b\in M_s$ with $s\in \N^n$, and let $a\in \overline{f(U_t)}$ with $t\in\N^{n+1}$. Then the set  $\overline{f(U_t)}\cap \overline{f(P)}$ is clopen in $\overline{f(P)}$ and it is disjoint from $M_s$.

To prove this, let us first show that if $a\in f(P)$ and $b\in \overline{f(P)}\setminus f(P)$, then there exists a clopen in $\overline{f(P)}$ set containing $a$ and missing $b$.

Indeed, for some $n\in\N$ there are sequences $s\in \N^n$, $t\in\N^{n+1}$ with  $b\in M_s$  and  $a\in \overline{f(U_t)}$. Let $V=\overline{f(U_t)}\cap \overline{f(P)}$.

Clearly, $V$ is closed in $\overline{f(P)}$. To see that it is also open in $\overline{f(P)}$, let $(x_i)_{i\in\N}$ be a convergent sequence of elements in $\overline{f(P)}\setminus V$ with $x=\lim\limits_{i\to\infty}x_i$. Let $k$ be the smallest natural number (possibly 0 but clearly not grater than $n$)  such that 
$$
x_i\in \bigcup_{j\neq t(k)}\overline{f(U_{t|k\smallfrown j})}
$$
for all but finitely many $i\in\N$ (here $t|k$ denotes $t|\{j\in\N: j<k \}$, in particular $t|0$ is the empty sequence).

 It follows that $x\in \bigcup\limits_{j\neq t(k)}\overline{f(U_{t|k\smallfrown j})}\cup M_{t|k}$, cf. (B5), and the latter set being disjoint from $\overline{f(U_t)}$, we conclude that $x\notin V$.

Thus $V$ is indeed clopen in $\overline{f(P)}$, $a\in V$ and $b\not\in V$, since $M_s\cap \overline{f(U_t)}=\emptyset$. 

Now, if $C$ is any continuum in $\overline{f(P)}$, the preceding observation shows that either $C\sub f(P)$ or $C\sub \bigcup\{ M_s: s\in\sk\}=M$, cf. (C4). Since $f(P)$ is a copy of the irrationals, hence zero-dimensional, and so is $M$, being the countable union 
of closed zero dimensional sets $M_s$, cf. \cite[Theorem 1.3.1]{E}, in both cases, $C$ must be a singleton.

\smallskip 

Having justified (8), we conclude that $\overline{f(P)}$  is zero-dimensional, cf. \cite[Theorem 1.4.5]{E}.  On the other hand, $P\notin I$, cf. (C3). In effect, for the zero-dimensional compactum $C=\overline{f(P)}$ we have $f^{-1}(C)\notin I$ and this contradiction with our assumptions
completes the proof of  part (ii) of Theorem \ref{main}.~\qed

\section{Proof of Theorem \ref{main} \hbox{\rm (iii)}}\label{proof of (iii)} Again the scheme of the proof
%, involving the compact-valued maps $\widecheck{f}_U$, 
is analogous to the ones in preceding sections.

Striving for a contradiction, suppose that  $f^{-1}(C)\in I$ for any  compactum $C$ in $Y$ with $\mu(C)<\infty$. 
%In particular, since $Y$ has no isolated points, it follows that $f^{-1}(y)\in I$ for $y\in Y$.

% and the key ingredient in the proof is a claim, similar to %the claim in Section \ref{proof of (ii)} Claim \ref{claim for (ii)}, where instead of demanding that the set $M$ is zero-dimensional, we declare that, given $\varepsilon>0$, $M$ can be chosen so that $\mu(M)<\varepsilon$.

Since the measure $\mu$ is $\s$-finite, there are compact sets $F_i$ in $Y$ with $\mu(F_i)<\infty$ and such that if we let
$H=Y\setminus \bigcup\limits_i F_i$, then
\begin{enumerate}
	\item[(1)] $\mu(H)=0$. 
\end{enumerate}
 
 We have $B\setminus \bigcup\limits_i f^{-1}(F_i)\notin I$ and using a theorem of Solecki \cite{s1}, we can find a non-empty $G_{\delta}$-set $G$ in $X$, $G\sub B \setminus \bigcup\limits_i f^{-1}(F_i)$, such that 
$V\notin I$ for any non-empty relatively open $V$ in $G$ and $f|G: G\w H$ is continuous.

\ms

A key element of our reasoning is the 
following counterpart of %the claim in Section \ref{proof of (i)}.
Claims
 \ref{claim for (i)}
 and 
 \ref{claimfor(ii)}.                                                                                                                                                    

\ms 

%{\bf Claim.}
\begin{claim}\label{claim for (iii)}
%{\bf Claim.}
Let $U$ be a  non-empty relatively open set in $G$ and let
 $\varepsilon >0$.  Then there exist   compacta   $L\sub \overline{U}$ and $M\sub\overline{f(U)}$ such that
\begin{enumerate}

\item[(2)] $L$ is boundary in $\overline{U}$ and $L\notin I$,

\item[(3)] $\mu(M)<\varepsilon$, 
%in $\overline{f(U)}$,

\item[(4)] $L\sub \widecheck{f}_U[M]$.
\end{enumerate}
	
\end{claim}

\smallskip 

%To verify this claim, we use the fact that $Y$ being countable dimensional, $Y$ has defined the small inductive transfinite dimension, and hence we can concentrate on a compactum $Y'\sub Y$ with the minimal transfinite dimension for which $f^{-1}(Y')\notin I$.

In order to prove the claim, we shall consider two cases.

\ms

{\sl Case 1. There exists $i$ such that $\widecheck{f}_U[F_i]\notin I$.}

\smallskip 

We can cover $F_i$ by finitely many compacta $M_0,\ldots, M_{n-1}$ with $\mu(M_j)<\varepsilon$ for each $j$. Then  for some $j$, $\widecheck{f}_U[M_j]\notin I$. We let $M=M_j$ and pick a  compactum 
$L\sub \widecheck{f}_U[M]$, not in $I$ and  boundary in $\overline{U}$.

\ms

{\sl Case 2. For all $i$,  $\widecheck{f}_U[F_i]\in I$.}

\smallskip 

Then we  pick a boundary in 
$\overline{U}$ compactum  $L\sub \overline{U}\setminus \bigcup\limits_i \widecheck{f}_U[F_i]$ not in $I$ and we let $M=\{y\in \overline{f(U)}: \widecheck{f}_U(y)\cap L\neq \emptyset \}$. Consequently,  $L\sub \widecheck{f}_U[M]$ and since $M$ is contained in $H$, $\mu(M)=0$, cf. (1). 

\ms

Having justified the claim, we can use Lemma \ref{main lemma} to define a generalized Hurewicz system $(U_s)_{s\in\sk}$, $(L_s)_{s\in\sk}$  with the associated family  $(M_s)_{s\in\sk}$ 
satisfying  for each $s\in\sk$ conditions (B1)--(B5) with $J_s$ being the collection of  compacta $M$ in $Y$ with $\mu(M)<\frac{1}{2^{e(s)}}$, where $e:\sk\w \N$ is a fixed bijection.

These conditions guarantee that the set $P$  determined by this system has  properties (C1)--(C5) (granted by part (C) of Lemma \ref{main lemma}) and, moreover,
 $\mu(\bigcup\{M_s: s\in\sk\})\leq 2$.  Since $f(P)\sub H$ and $\mu(H)=0$, cf. (1), it follows, cf. (C4), that $\mu(\overline{f(P)})\leq 2$. On the other hand, $P\notin I$ so in effect, for the compactum $C=\overline{f(P)}$ we have $\mu(C)<\infty$ but $f^{-1}(C)\notin I$ which contradicts our assumptions and ends the proof.~\qed

%be the   copy of the irrationals in $G$  . 

%\section{Proof of Theorem \ref{quotients}}\label{homo}

\section{Homogeneity notions related to $\sigma$-ideals}\label{homo}

% Following Zapletal \cite{zap}, \cite{zap1}, we shall say
 Recall, cf. Section \ref{sec:1}, that a \s-ideal $I$
% generated by compact sets
on a \replaced{Polish space }{compactum} $X$ is  homogeneous, if for each $E\in  Bor(X)\setminus I$ there exists a Borel map $f:X\w E$ such that $f^{-1}(A)\in I$, whenever $A\in I$ (cf. \cite{zap}, \cite{zap1}). Examples of homogeneous \s-ideals include, cf. \cite{zap}:
\begin{itemize}

	\item   the \s-ideal of countable subsets of $X$,

\item the \s-ideal generated by compact sets in the irrationals,
%Baire space $\baire$,

\item the \s-ideal of meager Borel sets in the Cantor set,

\item the \s-ideal of Lebesgue-null Borel sets in the Cantor set.

\end{itemize}

\subsection{The \s-ideal $I(dim)$}\label{I(dim)}

Let (cf. Section \ref{sec:1}), $I(dim)$ be the \s-ideal 
 of Borel sets in the Hilbert cube $[0,1]^{\N}$ that can be covered by countably many finite-dimensional compacta, and let, for a compactum $X\sub [0,1]^{\N}$, $I_X(dim)$
 be the \s-ideal $I(dim)$ restricted to $Bor(X)$.

The \s-ideal  $I(dim)$ 
  is  not homogeneous in a strong way. To see this, let $X$ be a Henderson compactum  in $[0,1]^{\N}$, cf. 7.1, and let $Y\sub [0,1]^{\N}$ be a countable-dimensional compactum not in $I(dim)$, cf. \cite[Example 5.1.7]{E}.
  
  Since $I_X(dim)$ is calibrated, cf. 7.1,  by  Theorem \ref{main}(ii), there is no Borel map $f:B\w Y$ with $B\in Bor(X)\setminus I_X(dim)$ such that $f^{-1}(A)\in I_X(dim)$, whenever $A\in I_Y(dim)$.
  
  Applying a theory developed by Zapletal \cite{zap}, one infers that forcings associated with the  collections $Bor(X)\setminus I_X(dim)$ and $Bor(Y)\setminus I_Y(dim)$, partially ordered by inclusion, are not equivalent, cf. \cite{zap}, the final part of Section 2.3.
  
  This answers Question 3.1 of Zapletal \cite{zap2} (a partial answer was given in \cite{p}).

\subsection{The \s-ideals $J_0(\mu)$, $J_f(\mu)$}\label{I(mu)} 

Given a Borel measure $\mu$ on a compactum $X$, let (cf. Section \ref{sec:1}) $J_0(\mu)$, $J_f(\mu)$ be the \s-ideals of Borel sets in $X$ that can be covered by countably many compact sets of $\mu$-measure zero,  or finite $\mu$-measure, respectively.

\begin{subproposition}\label{non-hom}
	Let $\mu$ be a semifinite  nonatomic Borel measure on a compactum $X$ with $\mu(X)>0$.
	
\begin{enumerate}
	\item[(i)] 	The \s-ideal $J_0(\mu)$ is not homogeneous.

	\item[(ii)] If, moreover, $\mu$ is not \s-finite (in particular, $X\not\in J_f(\mu)$) and there exists a Borel set $Y\not\in J_f(\mu)$  with $\mu(Y)<\infty$ and $\mu|K(X)$ is a Borel mapping on the hyperspace $K(X)$,
    then 
	%there is a compactum $Y$ in $X$, $Y\not\in J_f(\mu)$, such for any Borel map $f:X\w Y$  there is a compact set $C$ in $Y$ with $\mu(C)<\infty$ but $f^{-1}(C)\not\in J_f(\mu)$. 
	the \s-ideal $J_f(\mu)$ is not homogeneous.

\end{enumerate}

\end{subproposition}

	\begin{proof}[Proof]
	(i)	
	Pick $Y\in Bor(X)\setminus J_0(\mu)$ with $\mu(Y)=0$ and any Borel map $f:X\w Y$. 
	By the Lusin theorem, there is a compact set $K$ in $X$ with $\mu(K)>0$ such that $f|K$ is continuous. If $C=f(K)$, then $C\in J_0(\mu)$ but $f^{-1}(C)\not\in J_0(\mu)$.
	
	\ms
	
	(ii)
	Pick $Y\in Bor(X)\setminus J_f(\mu)$ with $\mu(Y)<\infty$.  Let $f:X\w Y$ be any Borel function. By \cite[Proposition 6.2]{p-z-2},  there is a
	compact set $K$ in $X$ with $K\not\in J_f(\mu)$ (even of  non-\s -finite $\mu$-measure) such that $f|K$ is continuous. If $C=f(K)$, then $C\in J_f(\mu)$ but $f^{-1}(C)\not\in J_f(\mu)$.

	\end{proof}

%Let us note that it may happen that a Borel measure $\mu$ on a compactum $X$ is \s-finite but still $X\not\in J_f(\mu)$. It turns out, however, that in this case 	the \s-ideal $J_f(\mu)$ can be homogeneous (see Proposition \ref{counterexample}). 
In contrast to (ii) above, we shall show in Proposition \ref{counterexample} that for some \s-finite measures $\mu$ on compacta $X$ with $X\not\in J_f(\mu)$, the \s-ideal  $J_f(\mu)$ can be homogeneous.

\ms 

Recall that $\lambda$ and  $\Haus^1$ denote
the Lebesgue measure on $[0,1]$ and the 1-dimensional Hausdorff measure on the Euclidean square $[0,1]^2$, respectively. It is well known that the measure $\Haus^1$ (restricted to Borel sets in $[0,1]^2$) is nonatomic, semifinite but not \s-finite and  $\Haus^1|K([0,1]^2)$ is a Borel map (cf. \cite{r}). Moreover, it is easy to construct a dense $G_{\delta}$ set $Y$ in 
$[0,1]^2$ of $\Haus^1$-measure zero. Consequently, $Y\not\in J_f(\Haus^1)$ since, non-empty open sets in $[0,1]^2$ having infinite $\Haus^1$-measure, the \s-ideal $J_f(\Haus^1)$ contains meager sets only. This leads to the following corollary of Proposition \ref{non-hom}.

	\begin{subcor}\label{Haus-inhom}
 The \s-ideals $J_0(\lambda)$, $J_0(\Haus^1)$ and $J_f(\Haus^1)$ are not homogeneous.
		
	\end{subcor}

\subsection{The partial orders $Bor(X)\setminus J_0(\mu)$ and  $Bor(X)\setminus J_f(\mu)$}\label{the PO's}

Let us now shift our attention from the \s-ideals
$J_0(\mu)$ and $J_f(\mu)$ to the collections of Borel sets $Bor(X)\setminus J_0(\mu)$ and 
$Bor(X)\setminus J_f(\mu)$, partially ordered by inclusion. The key step in the proof of Theorem \ref{quotients} is the following result.

\begin{subtheorem}\label{homogeneity-Lwow}\hfill\null
	
%	\begin{enumerate}
%		\item 
\mbox{\rm (i)}		There is a copy of the irrationals $P$ in $[0,1]^2$ such that

		\begin{itemize}

			\item $P\notin J_f(\Haus^1)$,

			\item if $\mu$ is any nonatomic Borel measure  on a compactum $X\not\in J_f(\mu)$
			such that every Borel set $B\not\in J_f(\mu)$ contains a Borel set 
			$C\not\in J_f(\mu)$ with $\mu(C)<\infty$, then for each $B\in Bor(X)\setminus J_f(\mu)$ there is a homeomorphic embedding $h: P\w B$ such that, for $A\sub P$, $A\in J_f(\Haus^1)$ if and only if $h(A)\in J_f(\mu)$. 
			
		\end{itemize}
		
%		\item 
\mbox{\rm (ii)}	There is a copy of the irrationals $P$ in $[0,1]$ such that

		\begin{itemize}
			\item $P\notin J_0(\lambda)$,

			\item for any semifinite nonatomic Borel measure $\mu$  on a compactum $X$ with $\mu(X)>0$ and
			for each $B\in Bor(X)\setminus J_0(\mu)$ there is a homeomorphic embedding $h: P\w B$ such that, for $A\sub P$, $A\in J_0(\lambda)$ if and only if $h(A)\in J_0(\mu)$.
			
		\end{itemize}

%	\end{enumerate}

\end{subtheorem}

%\begin{theorem}\label{homogeneity}	There is a copy of the irrationals $P$ in $[0,1]$ such that 
	%	\begin{itemize}
	%	\item $P\notin J_0(\lambda)$,
	%	\item for any semifinite nonatomic Borel measure $\mu$  on a compactum $X$ and for each $B\in Bor(X)\setminus J_0(\mu)$ there is a homeomorphic embedding $h: P\w B$ such that, for $A\sub P$, $A\in J_0(\lambda)$ if and only if $h(A)\in J_0(\mu)$.
%		\end{itemize}
%\end{theorem} 

%\begin{theorem}\label{homogeneity-Lwow}	There is a copy of the irrationals $P$ in $[0,1]^2$ such that
%	\begin{itemize}
%		\item $P\notin J_f(\Haus^1)$,
%		\item if $\mu$ is any nonatomic Borel measure  on a compactum $X\not\in J_f(\mu)$ such that every Borel set $B\not\in J_f(\mu)$ contains a Borel set $C\not\in J_f(\mu)$ with $\mu(C)<\infty$, then for each $B\in Bor(X)\setminus J_f(\mu)$ there is a homeomorphic embedding $h: P\w B$ such that, for $A\sub P$, $A\in J_f(\Haus^1)$ if and only if $h(A)\in J_f(\mu)$. 
%\end{itemize}
	
%\end{theorem}

%\begin{theorem}\label{homogeneity}
%	Let $\mu$ be a \s-finite nonatomic Borel measure on a compactum $X\not\in J_f(\mu)$. 

%	Then there is a copy of the irrationals $P$ in $X$ with the following property: $P\notin J_f(\mu)$ and for each $B\in Bor(X)\setminus J_f(\mu)$ there is a homeomorphic embedding $h: P\w B$ such that, for $A\sub P$, $A\in J_f(\mu)$ if and only if $h(A)\in J_f(\mu)$. 
%\end{theorem} 

\begin{proof}

%Since $X\notin J_f(\mu)$, removing from $X$ all open sets in $J_f(\mu)$ we get a compactum $X'$ none of whose non-empty relatively open sets is in $J_f(\mu)$. The assumptions about $\mu$ allow us, cf. the beginning of Section \ref{proof of (i)},
(i) Let $G$ be a copy of the irrationals  in $[0,1]^2$ which is $\Haus^1$-null and dense in $[0,1]^2$. Consequently, if $U$ is a  non-empty relatively open set in $G$, then $\overline{U}\notin J_f(\Haus^1)$ and since $\Haus^1(G)=0$ it follows that $\Haus^1(\overline{U}\setminus G)=\infty$. Hence
%, by a theorem of Gelbaum \cite{Ge}, 
there is a  Cantor set $L\sub \overline{U}\setminus G$ with %$1<\Haus^1(L)<\infty$ (with some extra effort
%, using a theorem of Oxtoby \cite{Ox},
%we could actually demand that
 $\Haus^1(L)=1$ (cf. \cite{r}).

This observation can be used  to define a generalized Hurewicz system $(U_s)_{s\in\sk}$, $(L_s)_{s\in\sk}$ such that, in particular, the following conditions are satisfied for each $s\in \sk$:

\begin{enumerate}

	\item[(1)] $L_s$ is a Cantor set with $\Haus^1(L_s)=1$,
	
	\item[(2)] $U_s$ is a non-empty relatively clopen subset of $G$,
%	\item $\overline{V(\s')}\cap \overline{V(\s'')}=\emptyset$ for distinct $\s', \s''\in \N^m$,
	
%	\item $V(\s^\smallfrown m)\sub V(\s)$, 	$\overline{V(\s^\smallfrown m)}\cap L(\s)=\emptyset$,
	
%	\item $L(\s)=\bigcap\limits_n\overline{\bigcup\limits_{m\geq n}V(\s^\smallfrown m)}$,
	
	\item[(3)] $\lim\limits_{i\to\infty}\hbox{diam}(U_{s^\smallfrown i})= 0$, $\hbox{diam}(U_s)\leq 2^{-length(s)}$, with respect to a  fixed complete metric on $G$.
	
%	\item $P=\bigcap\limits_m\bigcup\{\overline{V(\s)}:\s\in\N^m\}\sub G$.

\end{enumerate}

 These conditions   guarantee,  cf. Section \ref{sec:2},  that the copy of the irrationals $P$
 % = \bigcap_n \bigcup \{U_s: \hbox{length(s)} = n\}$$ 
 determined by this system has the following properties:  
 \begin{enumerate}
 	\item[(4)] $\overline{P} = P \cup \bigcup\{L_s: s\in\sk\}$ and $L_s\cap L_t=\emptyset$ for $s\neq t$,
 	
 	\item[(5)]  each nonempty relatively open subset of $\overline{P}$ contains infinitely many sets $L_s$.
 
 \end{enumerate}

In particular, by a Baire category argument,  $P\notin J_f(\Haus^1)$.

%This allows one to pick inductively, starting from $V(\emptyset)=G$, nonempty relatively clopen sets $V(\s)$ in $G$, $\s\in \N^{<\N}$, and Cantor sets $L(\s)\sub \overline{V(\s)}\setminus G$ such that

\smallskip 

Let us now consider an arbitrary $B\in Bor(X)\setminus J_f(\mu)$. By the properties of $\mu$ without loss of generality we can assume that $\mu(B)<\infty$. By a theorem of Solecki \cite{s1} we can first find a $G_\delta$ in $B$ not in $J_f(\mu)$ and then %properties of $\mu$, 
shrinking it further 
we can pick a copy of the irrationals $G'$ in $B$ with $\mu(G')=0$ such that for each non-empty relatively open set $U'$ in $G'$, \added{we have} \replaced{$U'\notin J_f(\mu)$ }{$\overline{U'}\notin J_f(\mu)$} so, in particular, $\mu(\overline{U'}\setminus G')=\infty$. 

A key element of our reasoning is the following observation.

\begin{subclaim}\label{claim for 2}
	Let $U$ and $U'$ be non-empty relatively open sets in $G$ and $G'$, respectively. Let 
	 $L\sub \overline{U}\setminus U$ be a  Cantor set with $\Haus^1(L)=1$. 
		
	Then there exist a Cantor set  $L'\sub \overline{U'}\setminus U'$ and a homeomorphism $g:U\cup L \w U'\cup L'$ such that $\mu(L')=1$ and
	$g|L$ is a measure preserving homeomorphism between $L$ and $L'$.  
	
\end{subclaim}

To prove the claim, we can first appeal to  results of Oxtoby \cite{Ox} to find a Cantor set 
 $L'\sub \overline{U'}\setminus U'$ and 
 a measure preserving homeomorphism $f:L\w L'$. 
 
 More precisely, let $\mathcal N$  denote    the  set  of the irrationals   in  $[0,  1]$,   and   let  $\lambda$
 be   the  restriction     of the  Lebesgue    measure  on $[0,  1]$ to  the   Borel   subsets   of
 $\mathcal N$. 
 %We can easily 
 Considering the product $P_1=L\times {\mathcal N}$, one can identify $L$ with a  subspace of  $P_1$, a copy of the irrationals equipped with a 
 %finite nonatomic
 Borel measure $\nu$ such that
 
 \begin{enumerate}
 	
 \item[(6)] $\nu(P_1)<\infty$,	
 
 \item[(7)] $\nu(\{x\})=0$ for each $x\in P_1$,
 
 \item[(8)] $\nu(U)>0$ for every non-empty open set in  $P_1$,
 
 \item[(9)] $\nu$ coincides with $\Haus^1$ on Borel sets in $L$.
 	
 \end{enumerate}
 
  In effect, by a theorem of Oxtoby \cite[Theorem 1]{Ox}, properties (6)--(8) guarantee that there is a homeomorphism $\varphi_1: {\mathcal N}\w P_1$ such that $\nu(\varphi_1(A))=\nu(P_1)\cdot\lambda(A)$ for any Borel set $A$ in $\mathcal N$.
 
 On the other hand, by theorems of Gelbaum \cite{Ge} and Oxtoby \cite[Theorem 2]{Ox}, there is  a copy of the irrationals $P_2$ in  $\overline{U'}\setminus U'$ with $\mu(P_2)=\nu(P_1)$ and a homeomorphism  
 $\varphi_2: {\mathcal N}\w P_2$ such that $\mu(\varphi_2(A))=\nu(P_1)\cdot\lambda(A)$ for any Borel set $A$ in $\mathcal N$.
 
 %To define now the desired Cantor set $L'\sub \overline{U'}\setminus U'$ and   a $\mu$-preserving homeomorphism $f:L\w L'$, 
  Now it suffices to let $L'=(\varphi_2\circ\varphi_1^{-1})(L)$ and
  $f=\varphi_2\circ\varphi_1^{-1}|L$ to obtain a desired Cantor set $L'$ and   a measure preserving homeomorphism $f:L\w L'$.
 
 \smallskip 
 
 Finally, since $U\cup L$ and $U'\cup L'$ are copies of the irrationals, a theorem of Pollard \cite{Po} provides an extension of $f$ to a homeomorphism
 $g:U\cup L \w U'\cup L'$ .

\ms

Having justified the claim, we can use it to define a generalized Hurewicz system $(U'_s)_{s\in\sk}$, $(L'_s)_{s\in\sk}$ of subsets of $G'$  together with homeomorphisms
\begin{enumerate}
	
\item[(10)]  $
 h_s: L_s\cup \bigcup_i U_{s\smallfrown i}\w
   L'_s\cup \bigcup_i U'_{s\smallfrown i}
   $,
\end{enumerate}    
   \noi
   satisfying the following conditions  for each $s\in \sk$:

\begin{enumerate}
	
	\item[(11)]
	$h_s(U_{s\smallfrown i})=U'_{s\smallfrown i}$,
	
	\item[(12)] $\mu(L'_s)=1$ and
	$h_s|L_s:L_s\w L'_s$ is measure preserving,

		\item[(13)] $U'_s$ is a non-empty relatively clopen subset of $G'$,
	
%	\item $\mu(L_s)>1$,
	
	%	\item $\overline{V(\s')}\cap \overline{V(\s'')}=\emptyset$ for distinct $\s', \s''\in \N^m$,
	
	%	\item $V(\s^\smallfrown m)\sub V(\s)$, 	$\overline{V(\s^\smallfrown m)}\cap L(\s)=\emptyset$,
	
	%	\item $L(\s)=\bigcap\limits_n\overline{\bigcup\limits_{m\geq n}V(\s^\smallfrown m)}$,
	
	\item[(14)] $\lim\limits_{i\to\infty}\hbox{diam}(U'_{s^\smallfrown i})= 0$ and
	 $\hbox{diam}(U'_s)\leq 2^{-length(s)}$, with respect to a  fixed complete metric on $G'$.
	
	%	\item $P=\bigcap\limits_m\bigcup\{\overline{V(\s)}:\s\in\N^m\}\sub G$.

\end{enumerate}

More precisely, we let $U'_{\emptyset}=G'$ and given $U'_s$, we select $L'_s$, $U'_{s\smallfrown i}$ and $h_s$  as follows. Claim \ref{claim for 2}
provides  a Cantor set  $L'_s\sub \overline{U'_s}\setminus U'_s$ and a homeomorphism $g_s:U_s\cup L_s \w U'_s\cup L'_s$ such that $\mu(L'_s)=1$ and
$g_s|L_s$ is a measure preserving homeomorphism between $L_s$ and $L'_s$. 
%Define $h_s|L_s= g_s|L_s$. 
For each $i\in\N$ let $W_i=g_s(U_{s\smallfrown i})$ and pick a non-empty
relatively clopen  set $U'_{s\smallfrown i}\sub W_i$ in $G'$ such that  $\hbox{diam}(U'_{
s\smallfrown i})\leq 2^{-(length(s)+i)}$ (with respect to a  fixed complete metric on $G'$). Since $W_i$ and $U'_{s\smallfrown i}$ are copies of  the irrationals, there are homeomorphisms $u_i: W_i\w U'_{s\smallfrown i}$ 
%with the help of which we finally define 
which give rise to a homeomorphism $h_s$, letting $h_s|L_s= g_s|L_s$ and
$h_s|U_{s\smallfrown i} = u_i\circ g_s|U_{s\smallfrown i}$.
	
	\smallskip 
		
Let $P'\sub G'\sub B$ be the copy of the irrationals determined by the	 system $(U'_s)_{s\in\sk}$, $(L'_s)_{s\in\sk}$.	Then, exactly as in the case of $P$, we have, cf. (3), (4), 			
			
 \begin{enumerate}
 	\item[(15)] $\overline{P'} = P' \cup \bigcup\{L'_s: s\in\sk\}$ and $L_s'\cap L_t'=\emptyset$ for $s\neq t$,
 	
 	\item[(16)]  %each nonempty relatively open subset of $\overline{P'}$ contains infinitely many sets $L'_s$ hence  
 	$P'\notin J_f(\mu)$.
 	 	
 \end{enumerate}

Note that if $e\in\baire$, both $\bigcap_m U_{e|m}$ and $\bigcap_m U'_{e|m}$ are singletons, which gives rise to a homeomorphism $h:P\w P'$, defined by letting
 \begin{enumerate}
\item[(17)] $h(x)\in \bigcap_m U'_{e|m}$ for  $x\in \bigcap_m U_{e|m}$.
 \end{enumerate}
For any $s\in\sk$, if $x\in P\cap U_{s\smallfrown i}$, both $h(x)$ and $h_s(x)$ belong to the set $U'_{s\smallfrown i}$ of diameter less than $2^{(-length(s)+i)}$. This, combined with (10), leads to the following observation: for any relatively closed set $C$ in $P$ and $s\in \sk$, 
$h_s(\overline{C}\cap L_s)= \overline{h(C)}\cap L'_s$, and in effect, $\Haus^1(\overline{C}\cap L_s)= \mu(\overline{h(C)}\cap L'_s)$.

Taking into account (4), (15) and the fact that $\Haus^1(P)=\mu(P')=0$, we conclude that for each relatively closed set $C$ in $P$, 
$\Haus^1(\overline{C})=\mu(\overline{h(C)})$ and this shows that for every $A\sub P$, $A\in J_f(\Haus^1)$ if and only if $h(A)\in J_f(\mu)$. 
%both $h$ and $h^{-1}$ take sets from $J_f(\Haus^1)$ to sets in $J_f(\mu)$, completing the proof of part (i).

\smallskip

(ii) We shall modify the proof of part (i) above in the following way. 
Let $G$ be a copy of the irrationals  in $[0,1]$ which is $\lambda$-null and dense in $[0,1]$. Consequently, if $U$ is a  non-empty relatively open set in $G$, then 
there is a  Cantor set $L\sub \overline{U}\setminus G$ with 
$\lambda(L)>0$. It follows that we may define a generalized Hurewicz system $(U_s)_{s\in\sk}$, $(L_s)_{s\in\sk}$ satisfying for each $s\in \sk$ conditions (1)--(5) with (1) replaced by $\lambda(L_s)>0$.  The copy of \added{the} irrationals $P$
determined by this system 
has properties (4)--(5)
and in effect, $P\notin J_0(\lambda)$.

If now  $B\in Bor(X)\setminus J_0(\mu)$, then by the properties of $\mu$ and a theorem of Solecki \cite{s1} 
we can pick a copy of the irrationals $G'$ in $B$ with $\mu(G')=0$ such that for each non-empty relatively open set $U'$ in $G'$, \added{we have} \replaced{$U'\notin J_0(\mu)$ }{$\overline{U'}\notin J_0(\mu)$} so, in particular, $\mu(\overline{U'}\setminus G')>0$. 

A  refinement of the proof of Claim \ref{claim for 2} leads to the following observation

\begin{subclaim}\label{claim for (ii)}
	Let $U$ and $U'$ be non-empty relatively open sets in $G$ and $G'$, respectively. Let 
	$L\sub \overline{U}\setminus U$ be a  Cantor set with $\lambda(L)>0$. 
	
	Then there exist a Cantor set  $L'\sub \overline{U'}\setminus U'$ and a homeomorphism $g:U\cup L \w U'\cup L'$ such that $\mu(L')>0$ and
	$g|L$ is a  homeomorphism between $L$ and $L'$ preserving measure up to a positive constant factor. In particular, for every $A\sub L$, $\lambda(A)=0$ if and only if $\mu(g(A))=0$.    
	
\end{subclaim}

We can now use the claim to define a generalized Hurewicz system $(U'_s)_{s\in\sk}$, $(L'_s)_{s\in\sk}$ of subsets of $G'$  together with homeomorphisms $h_s$ satisfying conditions (10)--(14) with (12) replaced by the requirements that 
$\mu(L'_s)>0$ and 
$h_s|L_s:L_s\w L'_s$ preserves measure up to a positive constant factor.

Arguing as before, we conclude that for each relatively closed set $C$ in $P$, 
$\lambda(\overline{C})=0$ if and only if $\mu(\overline{h(C)})=0$ and this shows that for any $A\sub P$, $A\in J_0(\lambda)$ if and only if $h(A)\in J_0(\mu)$, completing the proof of part (ii) and the proof of the theorem.

\end{proof}

Let us observe that the measure $\Haus^1$ itself  has the properties described in part (i) of Theorem \ref{homogeneity-Lwow}.

\begin{subremark}\label{haus}
 Every Borel set $B\not\in J_f(\Haus^1)$ contains a Borel set 
$C\not\in J_f(\Haus^1)$ with $\Haus^1(C)=0$.

\end{subremark}

\begin{proof}
By a theorem of Solecki \cite{s1} we  find a \added{non-empty} $G_\delta$ \deleted{sub}set $G$ in $B$ such that no non-empty relatively open set $U$ in $G$ is in $J_f(\Haus^1)$. Consequently, every element of $J_f(\Haus^1)$ below $G$  is meager in $G$ so it suffices to pick a dense $G_\delta$ subset $C$ of $G$ with $\Haus^1(C)=0$.

\end{proof}

\subsection{The proof of Theorem \ref{quotients}}\label{the proof of quotients}

%The next step towards the proof of Theorem \ref{quotients} is the following corollary of (the proof of) Theorem \ref{homogeneity-Lwow}.

To begin with let us make the following observation.

\begin{subproposition}\label{cellularity} Let $\mu$ be a 
	nonatomic Borel measure on a  compactum $X$. 
\begin{enumerate}
		\item[(i)] 
			Assume that every Borel set $B\not\in J_f(\mu)$ contains a Borel set 
			$C\not\in J_f(\mu)$ with $\mu(C)<\infty$.
			If  $B\in Bor(X)\setminus J_f(\mu)$ then there is a function $\varphi: \cantor\w Bor(B)\setminus J_f(\mu)$ such that
	  for any distinct $c,d \in \cantor$, $\varphi(c)\cap\varphi(d) = \emptyset$.
	 
	\item[(ii)] 
Assume that $\mu$ is semifinite. If $B\in Bor(X)\setminus J_0(\mu)$, then	 there is a function $\varphi: \cantor\w Bor(B)\setminus J_0(\mu)$ such that
for any distinct $c,d \in \cantor$, $\varphi(c)\cap\varphi(d) = \emptyset$.
	
\end{enumerate}

\end{subproposition}

\begin{proof}
To prove part (i), arguing as at the beginning of the proof of Theorem \ref{homogeneity-Lwow}(i) we  pick a copy of the irrationals $G$ in $B$ with $\mu(G)=0$ such that for each non-empty relatively open set $U$ in $G$, $\overline{U}\notin J_f(\mu)$.  This leads to a generalized Hurewicz system $(U_s)_{s\in\sk}$, $(L_s)_{s\in\sk}$ that determines a homeomorphic copy of the set $P\notin J_f(\mu)$. 

For each $c\in \cantor$ we let 
$$
S_c=\{s\in \sk: s(i)+c(i) \hbox{ is even for every}\ i<length(s) \}.
$$
and
$$
\varphi(c)=  \bigcap_n \bigcup \{U_s: s\in S_c \hbox{ and } \hbox{length(s)} = n\}.
$$

Thus $\varphi(c)$ 
may be viewed as the copy of the irrationals in $G$ determined by the system	
$(U_s)_{s\in S_c}$, $(L_s)_{s\in S_c}$. In particular, $\varphi(c)\not\in J_f(\mu)$ and, moreover,  $\varphi(c)\cap\varphi(d) = \emptyset$ for any distinct $c,d \in \cantor$.

\smallskip 

Part (ii) can be proved analogously along the lines of the first part of the  proof of Theorem  \ref{homogeneity-Lwow}(ii).

\end{proof}

We are now ready to complete the proof of 
 Theorem \ref{quotients}.
 
 \ms
 
 To prove part (i), let $P$ be a copy of the irrationals in $[0,1]^2$, the existence of which is guaranteed by Theorem \ref{homogeneity-Lwow}(i). Let $\mathbf{A}$ ($\mathbf{B}$) be the quotient Boolean algebra $Bor(P)/(J_f(\Haus^1)\cap Bor(P))$ ($Bor(X)/J_f(\mu)$, respectively).
 
  Let us note that if $B\in Bor(X)\setminus J_f(\mu)$ and $h: P\w B$ is a homeomorphic embedding such that, for $A\sub P$, $A\in J_f(\Haus^1)$ if and only if $h(A)\in J_f(\mu)$, then $h$ induces an isomorphisms from $\mathbf{A}$ onto 
 the quotient Boolean algebra $Bor(h(P))/(J_f(\mu)\cap Bor(h(P))$. Consequently,
  by Theorem \ref{homogeneity-Lwow}(i), the family ${\mathbf{C}}$ of all non-zero elements $\mathbf{c} \in \mathbf{B}$ such that the relative algebra $\mathbf{B}\restriction \mathbf{c}$ is isomorphic to  $\mathbf{A}$, is dense in $\mathbf B$. 
  
  Let $\overline{\mathbf{A}}$ ($\overline{\mathbf{B}}$) be the completion of  ${\mathbf{A}}$
  (${\mathbf{B}}$, respectively). It follows that for each $\mathbf{c} \in \mathbf{C}$  the relative algebra $\overline{\mathbf{B}}\restriction \mathbf{c}$ is isomorphic to  $\overline{\mathbf{A}}$ and $\mathbf{C}$ is dense in $\overline{\mathbf B}$. 
  Moreover, Proposition \ref{cellularity} implies that for any non-zero element $\mathbf{b} \in \overline{\mathbf{B}}$, the  algebra $\overline{\mathbf{B}}\restriction \mathbf{b}$ has cellularity continuum from which it follows, $\mathbf{C}$ being dense below  $\mathbf{b}$,  that $\overline{\mathbf{B}}\restriction \mathbf{b}$ is isomorphic to the product of continuum many isomorphic copies of the algebra 
 $\overline{\mathbf{A}}$ (cf. \cite[Proposition 6.4]{ko}). This shows that the algebra $\overline{\mathbf{B}}$ is homogeneous (cf. \cite[Definition 9.12]{ko}). In effect, since $\overline{\mathbf{B}}\restriction \mathbf{b}$ is isomorphic to  $\overline{\mathbf{A}}$ for some  non-zero  $\mathbf{b} \in \overline{\mathbf{B}}$,  the algebras $\overline{\mathbf{B}}$ and   $\overline{\mathbf{A}}$ are isomorphic. In view of Remark \ref{haus}, the same applies to the completion of the quotient Boolean algebra  $Bor([0,1]^2)/J_f(\Haus^1)$ which completes the proof of part (i) of Theorem \ref{homogeneity-Lwow}.
 
 \smallskip 
 
 To prove part (ii), we  follow closely the preceding argument, appropriately   applying Theorem \ref{homogeneity-Lwow}(ii). Thus the proof of Theorem \ref{homogeneity-Lwow} is completed. \qed

 \ms

In particular, the partial order 
$Bor([0,1]^2)\setminus J_f(\Haus^1)$ is forcing homogeneous, while the \s-ideal 
$J_f(\Haus^1)$ is not homogeneous, and the same is true if $J_f(\Haus^1)$ is replaced by $J_0(\lambda)$.
As already observed in Section \ref{sec:1}, it seems that examples illustrating this phenomenon did not appear in the literature, (cf.  \cite{zap}, comments following Definition 2.3.7). 

Finally, note that while, by Theorem \ref{homogeneity-Lwow}, the  completion of the quotient Boolean algebra  $Bor([0,1]^2)/J_f(\Haus^1)$ is homogeneous, the  algebra  $Bor([0,1]^2)/J_f(\Haus^1)$ itself is not, since by Sikorski's theorem \cite[15.C]{k}, this would imply the homogeneity of the \s-ideal $J_f(\Haus^1)$. The same is 
also true if $J_f(\Haus^1)$ is replaced by $J_0(\lambda)$.

%\begin{cor}	Let $\mu$ be a nonatomic Borel measure  on a compactum $X\not\in J_f(\mu)$	such that every Borel set $B\not\in J_f(\mu)$ contains a Borel set 	$C\not\in J_f(\mu)$ with $\mu(C)<\infty$. If $B\in Bor(X)\setminus J_f(\mu)$ , then  the completion of the quotient Boolean algebra $Bor(B)/J_f(\mu)$ is isomorphic to 	 the completion of the quotient Boolean algebra $Bor([0,1]^2)/J_f(\Haus^1)$. In particular, the completion of the quotient Boolean algebra $Bor(X)/J_f(\mu)$ is homogeneous.\end{cor}

%\begin{cor}	Let $\mu$ be a semifinite nonatomic Borel measure on a compactum $X$. If $B\in Bor(X)\setminus J_0(\mu)$ , then the completion of the quotient Boolean algebra $Bor(B)/J_0(\mu)$ is isomorphic to 	 the completion of the quotient Boolean algebra $Bor([0,1])/J_0(\lambda)$. In particular, the completion of the quotient Boolean algebra $Bor(X)/J_0(\mu)$ is homogeneous.\end{cor}

%\break

\section{Comments}\label{comments}

\subsection{Calibrated \s-ideals}\label{calibrated} 
If $X$ is a Henderson compactum, i.e., $dimX=\infty$ but $X$ contains no 1-dimensional 
subcompactum, cf. \cite[Example 5.2.23]{E}, then the \s-ideal $I_X(dim)$ is calibrated, cf. \cite{zap2}, \cite{p}. 

Also, the \s-ideal $J_{\sigma}(\H^1)$ of Borel subsets of the Euclidean square $[0,1]^2$ that 
can be covered by countably many compacta of \s-finite $\H^1$-measure is calibrated, cf. \cite{p-z-2}.

\subsection{The 1-1 or constant property of Sabok and Zapletal}\label{1-1 or constant}
From assertion (i) in Theorem \ref{main} it follows that any calibrated \s-ideal $I$ on a compactum $X$ has the following property: whenever $f:B\w\baire$ is a Borel map on $B\in Bor(X)\setminus I$ with all fibers in $I$, then there exists $C\in Bor(B)\setminus I$ on which $f$ is injective.

Indeed, the fact that this property can be derived from (i) was established by Sabok and Zapletal \cite{s-z} (the proof in \cite{s-z} is based on some forcing related arguments, and a justification in the realm of the classical descriptive set theory can be found in \cite{p-z-1}).

\subsection{Inhomogeneity of $J_f(\H^1)$}\label{inhomogeneity}
%Following Zapletal \cite{zap}, \cite{zap1}, we shall say that a \s-ideal $I$ generated by compact sets in $X$ is {\sl homogeneous}, if for each $Y\in  Bor(X)\setminus I$ there exists a Borel map $f:X\w Y$ such that $f^{-1}(A)\in I$, whenever $A\in I$.

%The \s-ideal $I(dim)$ on the Hilbert cube $[0,1]^{\N}$ is inhomogeneous in a strong way. To see this, let us consider a Henderson compactum $X$ in $[0,1]^{\N}$, cf. \ref{calibrated}, and a countable dimensional compactum $Y$ in $[0,1]^{\N}$, not in $I(dim)$. Then, by (ii) in Theorem \ref{main}, there is no Borel map $f:B\w Y$ with $B\in Bor(X)\setminus I(dim)$ and $f^{-1}(A)\in I(dim)$, whenever $A\in I(dim)$, as the \s-ideal $I(dim)$ is calibrated, cf. \ref{calibrated}.

As was already proved in Corollary \ref{Haus-inhom}, the \s-ideal $J_f(\H^1)$ on the Euclidean square $[0,1]^2$ is not homogeneous. Here is another proof of this fact. Let $Y\sub [0,1]^2$ be a compactum not in $J_f(\H^1)$ on which $\H^1$ is \s-finite, and let $f: [0,1]^2\w Y$ be any Borel function.
As was recalled in Section \ref{calibrated}, the \s-ideal $J_{\sigma}(\H^1)\supseteq J_f(\H^1)$
is calibrated in the square, and by (iii) in Theorem \ref{main}, there exists a compact set $C$ in $Y$  with $\H^1(C)<\infty$ and $f^{-1}(A)\notin J_{\sigma}(\H^1)$.

\subsection{Homogeneity of $J_f(\mu)$ for \s-finite $\mu$}\label{homogeneity}

%As for the general case, 
The following result shows that the requirement imposed on  $\mu$ to be {\sl non}-\s-finite cannot be dropped from the assumptions of  Proposition \ref{non-hom}(ii). 

\begin{proposition}\label{counterexample}
Let $\nu$ be a \s-finite nonatomic measure on a compactum $X$ such that all nonempty open sets have positive $\nu$-measure, and let $\mu$ be a nonatomic Borel measure on a compactum $Y\not\in J_f(\mu)$.
%such that every Borel set $B\not\in J_f(\mu)$ contains a Borel set $C\not\in J_f(\mu)$ with $\mu(C)<\infty$.

Then for any $B\in Bor(Y)\setminus J_f(\mu)$ with $\mu(B)<\infty$ there is a Borel map $f:X \w B$ such that,
whenever $A\in J_f(\mu)$, $f^{-1}(A)\in J_f(\nu)$.
 
\end{proposition} 

\begin{proof}
Let $P\sub B$ be \replaced{a }{the} copy of the irrationals defined as in 	the proof of Theorem \ref{homogeneity-Lwow}(i) and let us adopt the notation from that proof. \added{In particular, $\overline{P} = P \cup \bigcup\{L_s: s\in\sk\}$, the Cantor sets $L_s$ are pairwise disjoint and $\mu(L_s)=1$ for each $s\in \sk$.}

Let us note that  the compactum $\overline{P}$ is zero-dimensional as it contains no non-trivial continuum
%can be split into finitely many pairwise disjoint compacta of diameter less than $\varepsilon$
 (cf. the proof of (8) in Section \ref{proof of (ii)}). 
Since, moreover,
$\overline{P}\setminus P$ is dense in $\overline{P}$, removing a countable dense set from $\overline{P}\setminus P$, we get a copy of the irrationals $H$ such that
\begin{enumerate}
	
\item[(1)] $P\sub H\sub \overline{P}$,\quad 
$|\overline{P}\setminus H|\leq\aleph_0$.

\end{enumerate}	

\smallskip 
	
(A) Let us assume first that $X$ is a copy of the irrationals.

\smallskip 

The measure $\nu$ being  \s-finite and  nonatomic, \added{by a result of Gelbaum \cite{Ge}}, there are pairwise 	disjoint Cantor sets $C_0, C_1, \dots$ in $X$ with $\nu(C_n)\leq 1$ for each $n\in\N$ such that 
$\nu(X\setminus \bigcup_i C_i)=0$.

Let us fix a complete metric $d$ on $H$.

We shall define inductively homeomorphisms $h_n: X\w H$ such that  for each $n\in\N$
\begin{enumerate}
	
	\item[(2)]  $\nu(A)=\mu(h_n(A))$ for any Borel $A\sub C_n$,
	
	\item[(3)] $h_{n+1}|(C_0\cup\ldots\cup C_{n})=
	h_{n}|(C_0\cup\ldots \cup C_{n})$,
		
	\item[(4)] $d(h_{n+1}(x), h_n(x))\leq 2^{-n}$ for any $x\in X$.
	
\end{enumerate}	

To define $h_0$, using results of Oxtoby \cite{Ox},  we fix a homeomorphism $u: C_0\w u(C_0)\sub L_{\emptyset}\cap H$ such that  $\nu(A)=\mu(u(A))$ for any Borel $A\sub C_0$ (cf. the proof of \ref{claim for 2}), and  let $h_0$ be an extension of $u$ to  a homeomorphism from $X$ onto $H$ whose existence is guaranteed by a theorem of Pollard \cite{Po}.

Assume that $h_n$ is already defined, and let $\U$ be a disjoint cover of $H$ by relatively clopen sets of $d$-diameter $\leq 2^{-n}$.

For each $U\in\U$, we consider $V=h_n^{-1}(U)$ and the homeomorphism $h_{n+1}|V:V\w U$ will be defined as follows.

  On $T=(C_0\cup\ldots\cup C_{n})\cap V$ we let $h_{n+1}$ coincide with $h_{n}$. Since $h_{n}(T)$ is compact and nowhere dense in $H$, $U\setminus 
  h_{n}(T)$ is a nonempty relatively open subset of $H$, so  one can find $L_s$ such that $L_s\cap H \sub U\setminus 
	h_{n}(T)$ (cf. Section \ref{sec:2}). Using again results of Oxtoby \cite{Ox} and Pollard \cite{Po}, we first pick a homeomorphic embedding $w: C_{n}\cap V\w L_s\added{\cap H}$ such that $\nu(A)=\mu(w(A))$, for any Borel $A\sub C_n\cap V$, and then  extend $w$ to a homeomorphism $h_{n+1}|V: V\w U$. 
	
	Then, conditions (2), (3), (4) are met.
	
	Now, (4) guarantees that
	
	\begin{enumerate}
		
		\item[(5)] the sequence $(h_n)$ uniformly converges to a continuous function $g: X\w H$,
		
		\item[(6)] $g|C_n=h_n|C_n$ for $n\in \N$.

	\end{enumerate}	
	
From (2), (6) and the fact that the Cantor sets $g(C_n)$ are pairwise disjoint, cf. (3), we infer that for any Borel set $A\sub X$,
$$
\nu\Bigl(A\cap \bigcup_n C_n\Bigr)=\sum_n \nu(A\cap C_n)=\sum_n \mu(g(A\cap C_n))=\mu\Bigl(g(A\cap \bigcup_n C_n)\Bigr).
$$

Since $\nu(X\setminus \bigcup_n C_n)=0$, we conclude that
	\begin{enumerate}
	
			\item[(7)] $\nu(A)\leq \mu(g(A))$ for any Borel $A\sub X$.
		
		\end{enumerate}

	Let $A\in J_f(\mu)$ and assume that $A\sub \overline{P}$. Then $A\sub \bigcup_j F_j$, where $F_j\sub \overline{P}$ are closed and $\mu(F_j)<\infty$ for every $j\in\N$. It follows, by (5) and (7), that $g^{-1}(F_j)$ are closed sets of finite $\nu$-measure, and hence $g^{-1}(A)\in J_f(\nu)$.
	
	Since the range of $g$ may not be contained in $P$, we shall slightly correct $g$ to get a required map $f:X\w P\sub B$.
	
	Let $M=g^{-1}(\bigcup_s L_s)$. Then, as we have noticed, $M\in J_f(\nu)$. Now, we define $f:X\w P$ so that $f$ coincides with $g$ on $X\setminus M$ and takes $M$ to a point in $P$.
	
	\smallskip 
	
	(B) Now, let $\nu$ be a \s-finite nonatomic  Borel measure on a compactum $X$ such that nonempty open sets have positive $\nu$-measure.
	
	By a result of Gelbaum \cite{Ge},  there  is  a  countable open basis  $(U_n)$  of $X$   such that $\nu(\partial U_n) = 0$ for all  $n\in\N$,  where $\partial U_n$ denotes
	the  boundary  of  $U_n$.
	Then $L=\bigcup_n \partial U_n$ is a \s-compact set  in $X$ with $\nu(L)=0$ such that $X\setminus L$ is a copy of the irrationals.
	
	Let $B$ be a Borel set in $Y$ satisfying the assumptions.
	
	Using (A), we define a Borel map $f|(X\setminus L): X\setminus L \w B$ such that for any Borel $A\in J_f(\mu)$,
	
	$(f|(X\setminus L))^{-1}(A)\in J_f(\nu)$, and we let $f$ send $L$ to a point in $B$.
	
\end{proof}

\subsection{A calibrated \s-ideal which is not coanalytic}\label{calibrated not coanalytic}
If $E$ is a subset of a compactum $X$, $E\neq X$, the \s-ideal $K(E)$ is calibrated but need not be coanalytic.

However, we did not find in the literature examples of calibrated, non-coanalytic  \s-ideals $I$ on compacta $X$ with $\bigcup I=X$.

%The following example of such  \s-ideal requires the additional set-theoretic hypothesis that each uncountable coanalytic set contains a Cantor set. 

The following construction provides examples of such \s-ideals of arbitrary high complexity. 

\begin{proposition}
	Let $I$ be a calibrated \s-ideal on a compactum $X$. For each $A\sub [0,1]$ there exists a calibrated \s-ideal $J$ on $[0,1] \times  X$ generated by compact sets and a continuous function $\Phi: [0,1]\w K([0,1] \times  X)$ such that $A=\Phi^{-1}(J)$.
	
\end{proposition}

\begin{proof}

%Let
%\begin{enumerate}
%	\item[(1)] \A = \{K\in K([0,1] \times X): \} 
%\end{enumerate}

Let $J$ consist of Borel sets  in  
$[0,1] \times  X$ that can be covered by countably many compact sets $K$ with 
$K_t=\{x\in X: (t,x)\in K \}\in I$, for each $t\not\in A$. 

Since $I$ is calibrated  one readily checks that so is $J$. 

 The function $\Phi(t)=\{t \}\times X$, $t\in [0,1]$, is a continuous map from $[0,1]$ to $K([0,1] \times X)$ and it is clear  that $A=\Phi^{-1}(J)$.
 
 \end{proof}

%Let us pick a coanalytic set $E\sub\cantor$ such that 
%$\E=\{K\in K(\cantor): |K\cap E|\leq\aleph_0\}\notin\mathbf{\Sigma_2^1}$.
%is not coanalytic in $K(\cantor)$.

%Then the set-theoretic hypothesis we assumed guarantees that the \s-ideal $I$ on $\cantor$ generated by $\E$ is calibrated, but $I\cap K(\cantor)=\E\notin\mathbf{\Sigma_2^1}$. 
%is not coanalytic in $K(\cantor)$.

\subsection{Comparing $Bor(X)/J_0(\mu)$ and $Bor(X)/J_f(\mu)$ }\label{comparing algebras}\mbox{} 

(A) Let $\mu$ be a \s-finite nonatomic measure on a compactum $X$ with $X\notin J_f(\mu)$. Then the consequence of Theorem \ref{main} (iii) indicated in Section \ref{sec:1}, combined with some results of Zapletal \cite{zap} (indicated in Section \ref{sec:1} in the context of the \s-ideal $I(dim)$), show that the forcings associated with the partial orders $Bor(X)\setminus J_0(\mu)$ and  
$Bor(X)\setminus J_f(\mu)$ are not equivalent. 

In particular, neither of the quotient Boolean algebras 
$Bor(X)/J_0(\mu)$ and $Bor(X)/J_f(\mu)$ embeds densely into the completion of the other.

\smallskip 

(B) Let $\mu^h$ be a semifinite but not \s-finite Hausdorff measure on a compactum, associated with a continuous nondecreasing function $h: [0,+\infty]\w [0,+\infty]$ with $h(r)>0$ for $r>0$ and $h(0)=0$, cf. \cite{r}. Then $Bor(X)/ J_f(\mu^h)$ does not embed densely into $Bor(X)/J_0(\mu^h)$.

This was proved in \cite{p-z-2} under the additional assumption that the calibrated \s-ideal  $ J_0(\mu^h)$ has the 1-1 or constant property (cf. Section \ref{1-1 or constant}) but this is now granted by Theorem \ref{main} (i).

It is not clear, however, if 
 $Bor(X)/ J_f(\mu^h)$ can be embedded densely into the completion of $Bor(X)/J_0(\mu^h)$.

\bibliographystyle{amsplain}

\begin{thebibliography}{26}
	
\bibitem{ba-ju} T.~Bartoszy\'nski, H.~Judah, \textit{Set Theory. On the structure of the real line}, A K Peters 1995.	

%\bibitem{d-r} O.~Davies and C.~A.~Rogers, \textit{The problem of subsets of finite positive measure}, Bull. London Math. Soc., \textbf{1} (1969), 47–-54.

%\bibitem{d} G.~Debs. \textit{Polar \s-ideals of compact sets}, Trans.  Amer. Math. Soc. \textbf{347} (1995), 317-–338.

%\bibitem{e} G.~Edgar, \textit{Measure, Topology, and Fractal Geometry}, Undergraduate Texts in Mathematics, Springer-Verlag, 2008.
\bibitem{E}   R.~Engelking,   \textit{Theory of dimensions, finite and infinite},	Lemgo, Heldermann 1995.
        
\bibitem{f-z} I.~Farah, J.~Zapletal,
    \textit{ Four and more}, Ann. Pure Appl. Logic \textbf{140(1-3)}(2006), 3--39.
    
    \bibitem{Ge}   B.~R.~Gelbaum,   \textit{Cantor  sets  in  metric  measure  spaces},  Proc.  Amer.  Math.   Soc.  \textbf{24}
    (1970), 341--343.
 
    
%  \bibitem{f} D.~H.~Fremlin, \textit{Measure Theory (Vol 2)}, Torres Fremlin, 2009. 
  
 %  \bibitem{h} J.~D. Howroyd, \textit{On dimension and on the existence of sets of finite positive Hausdorff measure}, Proc. London Math. Soc., \textbf{70} (1995), 581-–604.  

 \bibitem{hu} W.~Hurewicz, \textit{Relativ perfekte Teile von Punktmengen und Mengen (A)},  Fund. Math., \textbf{12} (1928), 78-–109.  
 
\bibitem{k-s-z}  V.~Kanovei, M.~Sabok, J.~Zapletal, \textit{Canonical Ramsey Theory on Polish Spaces}, Cambridge Tracts in Mathematics 202, Cambridge University Press, 2013.
 

\bibitem{k} A.~S.~Kechris, \textit{Classical descriptive set theory},
Graduate Texts in Math. 156, Springer-Verlag, 1995.

\bibitem{k-l-w} A.~S.~Kechris, A.~Louveau, W.~H.~Woodin \textit{The Structure of \s-Ideals of Compact Sets}, Trans. Amer. Math. Soc. \textbf{301(1)} (1987),  263--288.

\bibitem{ko} S.~Koppelberg \textit{General Theory of Boolean Algebras}, in: \textit{Handbook of Boolean Algebras}, Vol. I, Ed. by J.D. Monk, North-Holland, 1989.


%\bibitem{ku} K.~Kuratowski, \textit{Topology, vol. II}, Academic Press and
% Polish Scientific  Publishers, Warsaw,  1968.

%\bibitem{l} D.~G.~Larman, \textit{On  Hausdorff  measure  in  finite-dimensional  compact  metric  spaces},  Proc. London  Math.  Soc. \textbf{17(3)} (1967),  193--206.

%\bibitem{m-z} \'{E}.~Matheron, M.~Zelen\'{y}, \textit{Descriptive set theory of families of small sets } Bull. of Symbolic Logic Volume \textbf{13(4)} (2007), 482--537.

\bibitem{Ox} J.C.~Oxtoby, \textit{Homeomorphic measures in metric spaces}, Proc. Amer. Math. Soc. \textbf{24} (1970), 419--423.


\bibitem{p} R.~Pol, \textit{Note on Borel mappings and dimension}, Top. and Appl. \textbf{195} (2015), 275--283.

\bibitem{p-z-1} R.~Pol, P.~Zakrzewski, \textit{
 On Borel mappings and \s-ideals generated by closed sets}, Adv. Math. \textbf{231}
(2012), 651--603.

\bibitem{p-z-2} R.~Pol, P.~Zakrzewski, \textit{On Boolean algebras related to \s-ideals generated by compact sets}, Adv. Math. \textbf{297} (2016), 196--213.

\bibitem{Po} J.~Pollard, \textit{On extending homeomorphisms on zero-dimensional spaces}, Fund. Math. \textbf{67} (1970), 39--48.

\bibitem{r} C.~A.~Rogers, \textit{Hausdorff Measures}, Cambridge University Press, 1970.

\bibitem{sa} M.~Sabok, \textit{Forcing, games and families of closed sets}, Trans. Amer. Math. Soc.  \textbf{364} (2012), 4011--4039

\bibitem{s-z} M.~Sabok, J.~Zapletal,
\textit{Forcing properties of ideals of closed sets}, J. Symbolic Logic \textbf{76} (2011), 1075–-1095.

\bibitem{s1} S.~Solecki, \textit{Covering analytic sets by families of closed sets},
Journal of Symbolic Logic  \textbf{59(3)} (1994), 1022--1031.

%\bibitem{s2} S.~Solecki, \textit{$G_\delta$ ideals of Compact Sets}, J. Eur. Math. Soc. \textbf{13} (2011), 853-–882.

\bibitem{zap} J.~Zapletal, \textit{Forcing idealized}, Cambridge Tracts in Mathematics 174,
Cambridge University Press, Cambridge, 2008.

\bibitem{zap1} J.~Zapletal, \textit{ Descriptive Set Theory and Definable Forcing}, Memoirs of American
	Mathematical Society. Providence, RI American Mathematical Society, 2004. 


\bibitem{zap2} J.~Zapletal, \textit{Dimension theory and forcing}, Top. and Appl. \textbf{167} (2014), 31-–35.





\end{thebibliography}

\end{document}